\documentclass[preprint,12pt,3p]{elsarticle}
\usepackage{amsmath,amsthm,amsfonts,amssymb}
\usepackage{fullpage}
\usepackage{subfig}
\usepackage{blkarray}
\usepackage{caption}
\usepackage{float}
\usepackage{algpseudocode}
\usepackage{tikz}
\usepackage{hyperref}
\usepackage{amsmath, amssymb}
\usepackage{graphicx}    
\usepackage{textcomp} 
\newcommand{\circleddagger}{\raisebox{0.2ex}{\tiny\textcircled{$\dagger$}}}
\hypersetup{
    colorlinks=true,
    linkcolor=blue,
    urlcolor=blue}

\numberwithin{equation}{section}
\usepackage{xcolor}
\usepackage{colortbl}
\usepackage[normalem]{ulem}
\usepackage{sectsty}

\definecolor{astral}{RGB}{46,116,181}
\subsectionfont{\color{astral}}
\sectionfont{\color{astral}}
\usepackage{colortbl}

\DeclareMathAlphabet{\mathpzc}{OT1}{pzc}{m}{it}
\DeclareFontFamily{OT1}{pzc}{}
\DeclareFontShape{OT1}{pzc}{m}{it}{<-> s * [0.900] pzcmi7t}{}
\DeclareMathAlphabet{\mathpzc}{OT1}{pzc}{m}{it}

\usepackage{amsbsy}
\usepackage{amsmath}
\usepackage{accents}
\newlength{\dhatheight}

\newenvironment{frcseries}{\fontfamily{frc}\selectfont}{}

\DeclareMathAlphabet\mathbfcal{OMS}{cmsy}{b}{n}
\definecolor{darkslategray}{rgb}{0.18, 0.31, 0.31}
\definecolor{warmblack}{rgb}{0.0, 0.26, 0.26}

\usepackage{multirow}
\usepackage{mathtools}
\usepackage{algorithm}
\usepackage{algorithmicx}
\makeatletter
\def\BState{\State\hskip-\ALG@thistlm}
\makeatother
\newtheorem{theorem}{Theorem}[section]
\newtheorem{lemma}[theorem]{Lemma}
\newtheorem{corollary}[theorem]{Corollary}
\theoremstyle{definition}
\usepackage{hyperref}
\newtheorem{definition}{Definition}[section]

\newtheorem{example}{Example}[section]

\journal{}

\begin{document}

\begin{frontmatter}

\title{ \textcolor{warmblack}{\bf On Characterizations of W-weighted DMP and MPD Inverses}}

\author[label1]{Rajesh Senapati}
\address[label1]{Department of Mathematics, School of Advanced Sciences,  Vellore Institute of Technology, Chennai, India}
\ead{rajesh.senapati2024@vitstudent.ac.in}

\author[label1]{Ashish Kumar Nandi \corref{cor1}}
\ead{ashishkumar.nandi@vit.ac.in}

\begin{abstract}
Recently, the weak Drazin inverse and its characterization have been crucial studies for matrices of index $k$. In this article, we have revisited W-weighted DMP and MPD inverses and constructed a general class of unique solutions to certain matrix equations. Moreover, we have generalized the W-weighted Drazin inverse of Meng [The DMP inverse for rectangular matrices, Filomat, 31(19), 6015–6019 (2017)] using the minimal rank W-weighted weak Drazin inverse. In addition to that, we have derived several equivalent properties of W-weighted DMP and MPD inverses for minimal rank W-weighted weak Drazin inverse of rectangular matrices. Furthermore, some projection-based results are discussed for the characterization of minimal rank W-weighted Drazin inverse, along with some new expressions that are derived for MPD and DMP inverses. Thereby, we have elaborated certain expressions of the perturbation formula for W-weighted weak MPD and DMP inverses.  As an application, we establish the reverse and forward order laws using the W-weighted weak Drazin inverse and the minimal rank W-weighted weak Drazin inverse, and apply these results to solve certain matrix equation.

\end{abstract}

\end{frontmatter}

\section{Introduction}
Consider $\mathbb{C}^{m\times{n}}$ as the collection of all complex matrices of size ${m \times{n}}$. For any $A\in\mathbb{C}^{m \times{n}}$, $A^*$, $\mathcal{R}(A)$, $\mathcal{N}(A)$, $\mathbf{O}$ and $P_A$ denote the conjugate transpose, range, null space, zero matrix, and orthogonal projection onto the range space of $A$, respectively. For $B\in\mathbb{C}^{m\times{n}}$, the Moore-Penrose inverse $H\in\mathbb{C}^{n\times{m}}$ is the unique solution to the system of equations $BHB=B$, $HBH=H$, $BH=(BH)^*$ and $HB=(HB)^*$ and is denoted as $B^{\dag}$ \cite{Philos, Penrose17}. If the matrix $H$ satisfies only one condition $BHB=B$, then it is called the inner inverse of $B$ and is represented as $B\{1\}$. Similarly, for the outer inverse of $B$, the matrix $H$ satisfies the condition $HBH=H$ and is indicated as $B\{2\}$. Recall that if $\mathcal{R}(H)=E$ and $\mathcal{N}(H)=F$, then $B^{(2)}_{E,F}$ is the unique outer inverse of $B$ \cite{Ben2}. For a square matrix $B\in\mathbb{C}^{n\times{n}}$, the smallest nonnegative integer $\tilde{k}$ with $rank(B^{\tilde{k}})=rank(B^{\tilde{k}+1})$ is called the index of the matrix $B$ and represented as $ind(B)$. Suppose $B\in\mathbb{C}^{n\times{n}}$ with $\tilde{k}=ind(B)$, the unique matrix $Y$ satisfies the conditions $YBY=Y$, $BY=YB$, $B^{\tilde{k}+1}Y=B^{\tilde{k}}$ is known as the Drazin inverse of $B$ and denoted by $B^D$ \cite{Drazin13}. For every square matrix $B\in\mathbb{C}^{n\times{n}}$, $X$ satisfying the condition $B^{\tilde{k}+1}X=B^{\tilde{k}}$ is referred to as the weak Drazin inverse of $B$ and if it satisfies the condition $rank(X)=rank(B^{\tilde{k}})$, then it is called the minimal rank weak Drazin inverse of $B$ \cite{Campbell20}, respectively. Similarly, if $X$ satisfying the condition $XB^{\tilde{k}+1}=B^{\tilde{k}}$ with $rank(X)=rank(B^{\tilde{k}})$, then it is called the minimal rank right weak Drazin inverse of $B$ \cite{Campbell20}. Let $B\in\mathbb{C}^{m \times{n}}$ and $\tilde{W}\in\mathbb{C}^{n \times{m}}\setminus\{0\}$ with $\tilde{k}=\max\{ind(B\tilde{W}),ind(\tilde{W}B)\}$, then the W-weighted Drazin inverse $B^{D,W}$ of $B$ is the uniquely determined solution to $B^{D,W}=B^{D,W}\tilde{W}B\tilde{W}B^{D,W}$, $B\tilde{W}B^{D,W}=B^{D,W}\tilde{W}B$ and $(B\tilde{W})^{\tilde{k}+1}B^{D,W}\tilde{W}=(B\tilde{W})^{\tilde{k}}$. It is well established that $B^{D,W}=[(B\tilde{W})^{D}]^2B=B[(\tilde{W}B)^D]^2$ \cite{Cline3}. For $B\in\mathbb{C}^{m \times{n}}$, $\tilde{W}\in\mathbb{C}^{n \times{m}}$ with $\tilde{k}=ind(B\tilde{W})$, if $X\in\mathbb{C}^{m \times{n}}$ satisfying the condition $X\tilde{W}(B\tilde{W})^{\tilde{k}+1}=(B\tilde{W})^{\tilde{k}}$ is known as the W-weighted weak Drazin inverse of $B$ \cite{Stank15}. If $\tilde{k}=ind(\tilde{W}B)$ with $Z$ satisfying the equation $\tilde{W}(B\tilde{W})^{\tilde{k}+1}Z=(\tilde{W}B)^{\tilde{k}}$ is called the W-weighted right weak Drazin inverse of $B$ \cite{Stank15}. The W-weighted weak Drazin inverse and W-weighted right weak Drazin inverse with the rank condition $rank(X)=rank((B\tilde{W})^{\tilde{k}})$ and $rank(Z)=rank((\tilde{W}B)^{\tilde{k}})$ is called the minimal rank W-weighted weak Drazin and minimal rank W-weighted right weak Drazin inverse of $B$ \cite{mos1}, respectively. Suppose $B\in\mathbb{C}^{m \times {n}}$, $\tilde{W}\in\mathbb{C}^{n \times {m}}$ with $\tilde{k}=\max\{ind(B\tilde{W}),ind(\tilde{W}B)\}$, then $Y=\tilde{W}B^{D,W}\tilde{W}BB^{\dag}=B^{D,\dag,W}\in\mathbb{C}^{n \times {m}}$ is the unique solution to the following systems
\begin{eqnarray}
    YBY=Y,~~YB=\tilde{W}B^{D,W}\tilde{W}B~~\mbox{and}~~(\tilde{W}B)^{\tilde{k}+1}Y=(\tilde{W}B)^{\tilde{k}}B^{\dag}.\label{First}
\end{eqnarray} is known as the W-weighted DMP inverse of $B$ \cite{Meng11}. Moreover, the dual W-weighted DMP inverse which satisfies the equations, 
\begin{eqnarray}
YBY=Y,~~BY=B\tilde{W}B^{D,W}\tilde{W}~~\mbox{and}~~(B\tilde{W})^{\tilde{k}+1}Y=(B\tilde{W})^{\tilde{k}}B^{\dag}\label{Second}
\end{eqnarray} 
is referred to as the W-weighted MPD inverse of $B$ and uniquely denoted by $Y=B^{\dag}B\tilde{W}B^{D,W}\tilde{W}=B^{\dag,D,W}$  \cite{Kyrchei12}. If $\tilde{W}=I$, then the W-weighted DMP and MPD coincide with DMP and MPD inverse of $B$ \cite{Malik22}. \par

Suppose $B\in\mathbb{C}^{m \times{n}}$, $\tilde{W}\in\mathbb{C}^{n \times{m}}$ with $\tilde{k}=ind(B\tilde{W})$, the W-weighted core-EP inverse $B^{\circleddagger,W}$ of $B$ is  the uniquely determined solution of the equations $\tilde{W}B\tilde{W}B^{\circleddagger,W}=(\tilde{W}B)^{\tilde{k}}[(\tilde{W}B)^{\tilde{k}}]^{\dag}$ and $\mathcal{R}((B\tilde{W})^{\tilde{k}})=\mathcal{R}(B^{\circleddagger,W})$. It is represented as $B^{\circleddagger,W}=B[(\tilde{W}B)^{\circleddagger}]^2=B^{\circleddagger,W}\tilde{W}B\tilde{W}A^{\circleddagger,W}=B\tilde{W}B^{\circleddagger,W}\tilde{W}B^{\circleddagger,W}$ \cite{Ferreyra6,mos7}. When $\tilde{W}=I$, then this inverse equals to $B^{\circleddagger}$  \cite{Core}. If $m\in\mathbb{N}$, $B\in\mathbb{C}^{m \times {n}}$, $\tilde{W}\in\mathbb{C}^{n \times {m}}$ with $\tilde{k}=\max\{ind(B\tilde{W}),ind(\tilde{W}B)\}$, the W-weighted m-WGI is defined as $B^{\textcircled{w}_m,W} = (B^{\circleddagger,W}\tilde{W})^{m+1}(B\tilde{W})^{m-1}B$  \cite{mos8}. This serves as the unique solution to the matrix equations $B\tilde{W}H_c=(B^{\circleddagger,W})^m(B\tilde{W})^{m-1}B$ and $B\tilde{W}X_c\tilde{W}H_c=X_c$. When $\tilde{W}=I$, this inverse is identical to the m-WGI, given by $B^{\textcircled{w}_m}=(B^{\circleddagger})^{m+1}B^m$ \cite{New, Ring}. If $m=1$, then $B^{\textcircled{w}_m}$ coincides with the weak group inverse $B^{\textcircled{w}}=(B^{\circleddagger})^2B$ \cite{2018}. Let $m\in\mathbb{N}$, $B\in\mathbb{C}^{m \times {n}}$, and $\tilde{W}\in\mathbb{C}^{n \times {m}}$ with $\tilde{k}=\max\{ind(B\tilde{W}),ind(\tilde{W}B)\}$, the W-weighted m-weak core inverse of $B$, denoted by $H_c=B^{\textcircled{\#}_m,W}=B^{\textcircled{w}_m,W}P_{(WB)^m}$ \cite{Ferreyra9} and it is uniquely determined by the systems $B\tilde{W}H_c=B(\tilde{W}B)^{\textcircled{\#}_m}$ and $B\tilde{W}H_c\tilde{W}H_c=H_c$. By using these inverses, several inverses have been established. For $B\in\mathbb{C}^{m \times{n}}$ and $\tilde{W}\in\mathbb{C}^{n \times{m}}\setminus\{0\}$, with $\tilde{k}=\max\{ind(B\tilde{W}),ind(\tilde{W}B)\}$, the W-MPCEP is defined as $B^{\dag,\circleddagger,W} = B^{\dag}A\tilde{W}B^{\circleddagger,W}\tilde{W}$ \cite{mos5}, the W-CEPMP $B^{\circleddagger,\dag,W}$ is given by $B^{\circleddagger,\dag,W} = \tilde{W}B^{\circleddagger,W}\tilde{W}BB^{\dag}$ \cite{mos5}, the W-m-WGMP inverse can be expressed as $B^{\textcircled{w}_m,W,\dag}=\tilde{W}B^{\textcircled{w}_m,W}\tilde{W}BB^{\dag}$ \cite{Gao14} . \par

All these generalized inverses play a crucial role in several applications. For example, the Moore-Penrose inverse solves the minimum-norm least-squares solution of linear systems \cite{Ben2}, the Drazin inverse has applications in Markov chains, singular differential equations, and so on \cite{Ben2, Campbell}. Moreover, the weak Drazin inverse and minimal rank weak Drazin inverse \cite{Campbell20} can be used due to their easier computation compared to the Drazin inverse. The minimal rank weak Drazin inverse \cite{Wu} is a generalization of inverses such as the Drazin inverse \cite{Drazin13}, the core-EP inverse \cite{Core}, the DMP inverse \cite{Malik22}, the weak group inverse \cite{2018}, the weak core inverse \cite{2021}. Its extension to rectangular matrices, known as the minimal rank W-weighted weak Drazin inverses \cite{mos1}, also generalizes several well-known inverses, including the W-weighted Drazin inverse  \cite{Cline3}, the W-weighted core-EP inverse \cite{Ferreyra6, mos7}, and the W-weighted-m-WGI \cite{mos8}. Moreover, if $X$ denotes the minimal rank weak Drazin inverse of $B$, then the matrix $Y=B^{\dag}XB$ uniquely satisfies \[
YBY=Y, BY=XB~\mbox{and}~YB^{\tilde{k}}=B^{\dag}B^{\tilde{k}},
\] is known as weak MPD inverse and suppose $Z$ is the minimal rank right weak Drazin inverse of $B$, then the matrix $Y=BZB^{\dag}$ is the unique matrix satisfying \[~YBY=Y,~YB=BZ~\mbox{and}~B^{\tilde{k}}Y=B^{\tilde{k}}B^{\dag}\] is called as weak DMP inverse \cite{mos16}. For $X=A^D$ and $Z=A^D$, weak MPD and DMP inverse reduces to MPD and DMP inverses \cite{Malik22}.\par

Inspired by the concepts of weak MPD and DMP inverses \cite{mos16}, along with minimal rank W-weighted weak Drazin inverses \cite{mos1}, we extend the concept of weak MPD and DMP inverses from square matrices to rectangular matrices. Since the W-weighted Drazin inverse \cite{Cline3} is a particular case of the minimal rank W-weighted weak Drazin inverse \cite{mos1}, it can used in the place of the W-weighted Drazin inverse \cite{Cline3}. Our primary objective is to investigate the solvability of the matrix equations (\ref{First}) and (\ref{Second}) using minimal rank W-weighted weak Drazin inverse and minimal rank W-weighted right weak Drazin inverse. In this framework, we present the weaker version of W-weighted MPD and DMP inverses. \par
Furthermore, we present a new type of characterization for the W-weighted weak MPD and DMP inverses and establish characterizations of the W-MPCEP and W-m-MPWG inverses, together with their dual forms, in the same framework. By expressing the W-weighted weak MPD inverse in terms of well-known generalized inverses, we derive several new results. We then carry out a perturbation analysis of the W-weighted weak MPD and DMP inverses. In addition, we examine the reverse and forward order laws for the minimal rank W-weighted weak Drazin inverse. Moreover, we extend this to study the triple reverse and forward order laws for the minimal rank W-weighted weak Drazin inverse. Finally, we employ the minimal rank W-weighted weak Drazin inverse to obtain the corresponding reverse order law for the W-weighted weak MPD inverse. As an implementations, a new class of matrix equations are solved by applying reverse and forward order laws.\par
This article is organized as follows. Section \ref{sub2} focuses on the characterizations of minimal rank W-weighted weak Drazin inverse and minimal rank W-weighted right weak Drazin inverse. Section \ref{sub3} is divided into three subsections. The first subsection discusses the solvability of the new type of extended systems and the developed characterizations, which present the expressions of W-weighted weak MPD and DMP inverses. In the second subsection, we investigate perturbation analysis of W-weighted weak MPD and DMP inverses. In the last subsection, we examine the reverse order law of W-weighted weak MPD inverse as well as the reverse, forward order laws of minimal rank W-weighted weak Drazin inverses.
\section{Some Preliminaries Inputs} 
We have now presented few initial results based on the minimal rank W-weighted weak Drazin inverse and the right weak Drazin inverse, whose equivalences are frequently used in the main results.  

\label{notp1.1} \label{sub2}

\begin{theorem}[Theorem $2.1$, \cite{mos1}]\label{main1.1} \normalfont
      Suppose $B\in\mathbb{C}^{m\times{n}}$, $\tilde{W}\in C^{n\times{m}}\setminus\{0\}$ with $\tilde{k} = ind(B\tilde{W})$ and $X\in\mathbb{C}^{m\times{n}}$, then the followings are equivalent:\\
      (i)  $X$ is the minimal rank ${W}$-weighted weak Drazin inverse of $A$, i.e., $X\tilde{W}(B\tilde{W})^{\tilde{k}+1}=(B\tilde{W})^{\tilde{k}}$ and $rank(X) = rank((B\tilde{W})^{\tilde{k}})$ ;\\
      (ii) $X\tilde{W}(B\tilde{W})^{\tilde{k}+1} = (B\tilde{W})^{\tilde{k}}$ and $\mathcal{R}(X) = \mathcal{R}((B\tilde{W})^{\tilde{k}})$ ;\\
      (iii) $X\tilde{W}B\tilde{W}X =X$ and $\mathcal{R}(X)=\mathcal{R}((B\tilde{W})^{\tilde{k}})$, i.e., $X\in{B}{\{2\}}_{{\mathcal{R}((B\tilde{W})^{\tilde{k}})},_*}$ ;\\
      (iv) $X = (B\tilde{W})^{\tilde{k}}[(B\tilde{W})^{\tilde{k}}]^{\dag}X$ and $X\tilde{W}(B\tilde{W})^{\tilde{k}+1} =(B\tilde{W})^{\tilde{k}}$ ;\\
      (v) $X\tilde{W}B\tilde{W}X =X$, $X=(B\tilde{W})^{\tilde{k}}[(B\tilde{W})^{\tilde{k}}]^{\dag}X$ and $X\tilde{W}(B\tilde{W})^{\tilde{k}+1} =(B\tilde{W})^{\tilde{k}}$ ;\\
      (vi) $B\tilde{W}X\tilde{W}X =X$ and $X\tilde{W}(B\tilde{W})^{{\tilde{k}}+1} =(B\tilde{W})^{\tilde{k}}$ ;\\
      (vii) $X = (B\tilde{W})^D(B\tilde{W})X$ and $X\tilde{W}(B\tilde{W})^{{\tilde{k}}+1} =(B\tilde{W})^{\tilde{k}}$.
 \end{theorem}
 \begin{theorem}[Theorem $2.8$, \cite{mos1}]\label{main1.2} \normalfont
 Suppose $B\in\mathbb{C}^{m\times{n}}$, $\tilde{W}\in C^{n\times{m}}\setminus\{0\}$ with $\tilde{k} = ind(\tilde{W}B)$ and $Z\in\mathbb{C}^{m\times{n}}$, then the followings are equivalent:\\
     (i) $Z$ is the minimal rank ${W}$-weighted right weak Drazin inverse of $B$, i.e., $W(BW)^{\tilde{k}+1}Z=(\tilde{W}B)^{\tilde{k}}$ and $rank(Z)=rank((\tilde{W}B)^{\tilde{k}})$ ;\\
     (ii) $\tilde{W}(B\tilde{W})^{\tilde{k}+1}Z=(\tilde{W}B)^{\tilde{k}}$ and $\mathcal{N}(Z)=\mathcal{N}((\tilde{W}B)^{\tilde{k}})$ ;\\
     (iii) $Z\tilde{W}B\tilde{W}Z=Z$ and $\mathcal{N}(Z)=\mathcal{N}((WB)^{\tilde{k}})$, i.e., $Z\in{B}{\{2\}}_{*,{\mathcal{N}((\tilde{W}B)^{\tilde{k}})}}$ ;\\
    (iv) $Z=Z[(\tilde{W}B)^{\tilde{k}}]^{\dag}(\tilde{W}B)^{\tilde{k}}$ and $\tilde{W}(B\tilde{W})^{\tilde{k}+1}Z=(\tilde{W}B)^{\tilde{k}}$ ;\\
    (v) $Z\tilde{W}B\tilde{W}Z=Z, Z=Z[(\tilde{W}B)^{\tilde{k}}]^{\dag}(\tilde{W}B)^{\tilde{k}}$ and $\tilde{W}(B\tilde{W})^{\tilde{k}+1}Z=(\tilde{W}B)^{\tilde{k}}$ ;\\
    (vi) $Z\tilde{W}Z\tilde{W}B=Z$ and $\tilde{W}(B\tilde{W})^{\tilde{k}+1}Z=(\tilde{W}B)^{\tilde{k}}$ ;\\
    (vii) $Z=Z(\tilde{W}B)^{D}(\tilde{W}B)$ and $\tilde{W}(B\tilde{W})^{\tilde{k}+1}Z=(\tilde{W}B)^{\tilde{k}}$ .
 \end{theorem}

\section{Main Results} \label{sub3}
We begin by presenting the definitions of the W-weighted weak MPD and DMP inverses. We then investigate the conditions under which certain systems of matrix equations are solvable. These matrix equations are characterized by the minimal rank W-weighted weak Drazin inverse and the minimal rank W-weighted right weak Drazin inverse. Next, we establish characterizations of the W-weighted weak MPD and DMP inverses and obtain their explicit formulas. In addition, we address the perturbation analysis and the reverse order law associated with the W-weighted weak MPD inverse. Finally, we derive both reverse and forward order laws for minimal rank W-weighted weak Drazin inverses.
\begin{definition}
 Let us consider the matrix $B\in\mathbb{C}^{m\times{n}}$ and  $\tilde{W}\in\mathbb{C}^{n\times{m}}$ such that $\tilde{k}=ind(A\tilde{W})$. Suppose that $Y$ is an arbitrary matrix and have fixed minimal rank W-weighted weak Drazin inverse of $B$, then the W-weighted weak MPD inverse of $B$ is defined by
 \begin{center}
     $B_W^{{\dag},D_w}=B^{\dag}B\tilde{W}Y\tilde{W}.$
 \end{center}
\end{definition}
\begin{definition}
Suppose $B\in\mathbb{C}^{m\times{n}}$. Assume that $\tilde{W}\in\mathbb{C}^{n\times{m}}$ with $\tilde{k}=ind(\tilde{W}B)$. Consider $Z$ is arbitrary with fixed minimal rank W-weighted right weak Drazin inverse of $B$, then the W-weighted weak DMP inverse of $B$ is defined by
\begin{center}
    $B_W^{{D_w,{\dag}}}=\tilde{W}Z\tilde{W}BB^{\dag}.$
\end{center}
\end{definition}
\begin{theorem}\label{Inverse} \normalfont
Consider the matrix $B\in\mathbb{C}^{m\times{n}}$ and  $\tilde{W}\in\mathbb{C}^{n\times{m}}$ such that $\tilde{k}=ind(B\tilde{W})$. Let $X$ denote the minimal rank $W$-weighted weak Drazin inverse of $B$. Then $Y = B^{\dag}B\tilde{W}X\tilde{W} $ is the unique solution to
\begin{center}
    $YBY=Y$, $BY=B\tilde{W}X\tilde{W}$ and $Y(B\tilde{W})^{\tilde{k}+1} =B^{\dag}(B\tilde{W})^{\tilde{k}+1}$.
\end{center}
\begin{proof}
Since $X$ is a minimal rank $W$-weighted weak Drazin inverse of $B$, we have $X=(BW)^{\tilde{k}}[(BW)^{\tilde{k}}]^{\dag}X$ and $(BW)^{\tilde{k}}=XW(BW)^{\tilde{k}+1}$. By applying Theorem \ref{main1.1}, we now proceed to verify that the matrix $Y=B^{\dag}B\tilde{W}X\tilde{W}$ satisfies the following significant equations.
\begin{eqnarray*}
YBY&=&B^{\dag}B\tilde{W}X\tilde{W}BB^{\dag}B\tilde{W}X\tilde{W} \\
&=&B^{\dag}B\tilde{W}X\tilde{W}B\tilde{W}X\tilde{W}\\
&=&B^{\dag}B\tilde{W}X\tilde{W}\\
&=&Y.
\end{eqnarray*}
Further, $BY = BB^{\dag}B\tilde{W}X\tilde{W} = B\tilde{W}X\tilde{W}$ since $BB^{\dag}B = B$. Again, $ Y(B\tilde{W})^{\tilde{k}+1} = B^{\dag}B\tilde{W}X\tilde{W}(B\tilde{W})^{\tilde{k}+1} = B^{\dag}B\tilde{W}(B\tilde{W})^{\tilde{k}} = B^{\dag}(B\tilde{W})^{\tilde{k}+1}$.
Now to establish the uniqueness, we have
\begin{eqnarray*}
     Y&=&YBY\\
     &=&YB\tilde{W}X\tilde{W}\\
     &=&YB\tilde{W}(B\tilde{W})^k[(B\tilde{W})^{\tilde{k}}]^{\dag}X\tilde{W}\\
     &=&Y(B\tilde{W})^{{\tilde{k}}+1}[(B\tilde{W})^{\tilde{k}}]^{\dag}X\tilde{W}\\
     &=&B^{\dag}(B\tilde{W})^{\tilde{k}+1}[(BW)^{\tilde{k}}]^{\dag}X\tilde{W}\\
     &=& B^{\dag}B\tilde{W}(B\tilde{W})^{\tilde{k}}[(B\tilde{W})^{\tilde{k}}]^{\dag}X\tilde{W}\\
     &=& B^{\dag}B\tilde{W}X\tilde{W}\\
     &=&Y.
\end{eqnarray*}
\end{proof}
\end{theorem}
\begin{lemma} \normalfont \label{Inverse 3.2}
Suppose $B\in\mathbb{C}^{m\times{n}}$ and  $\tilde{W}\in\mathbb{C}^{n\times{m}}$ with $\tilde{k}=ind(\tilde{W}B)$. If $Z$ is a minimal rank W-weighted right weak Drazin inverse of $B$, then $Y_1 = \tilde{W}Z\tilde{W}BB^{\dag}$ is the unique solution to  
 \begin{center}
  $Y_1BY_1=Y_1$, $Y_1B=\tilde{W}Z\tilde{W}B$ and $(\tilde{W}B)^{\tilde{k}+1}Y_1 =(\tilde{W}B)^{\tilde{k}+1}B^{\dag}$.
 \end{center}
\end{lemma}
Now for any matrix $A\in\mathbb{C}^{m\times{n}}$, we will discuss the generalization of $X$, for which the minimal rank W-weighted weak Drazin inverse reduces to the following standard generalized inverses. 
\begin{itemize}
    \item If $X=A^{D,W}$ \cite{Cline3}, then  $Y=A^{\dag}A\tilde{W}X\tilde{W}=A^{\dag}A\tilde{W}A^{D,W}\tilde{W}=A^{{\dag},D,W}$ \cite{Kyrchei12}. This effectively reduces W-weighted weak MPD to W-weighted MPD.
    \item If $X=A^{\circleddagger,W}$ \cite{Ferreyra6,mos7}, then we obtain $Y=A^{\dag}A\tilde{W}A^{\circleddagger,W}\tilde{W}=A^{\dag,\circleddagger,W}$ \cite{mos5}. It concludes that W-weighted weak MPD reduces to W-MPCEP.
    \item Suppose $X=A^{\textcircled{w}_m,W}$ \cite{mos8}, then we have $Y=A^{\dag}A\tilde{W}A^{\textcircled{w}_m,W}\tilde{W}=A^{\dag,\textcircled{w}_m,W}$ \cite{Gao14}. This means W-weighted weak MPD coincides with dual of the W-m-MPWG inverse.
\end{itemize}
The mentioned generalized inverses satisfy the uniqueness conditions as specified in Theorem \ref{Inverse}.
\begin{itemize}
    \item If $X=A^{D,W}$ \cite{Cline3}, then $Y=A^{\dag}A\tilde{W}A^{D,W}\tilde{W}$ \cite{Kyrchei12} is the unique solution of the systems:
\[
YAY=Y, AY=A\tilde{W}A^{D,W}\tilde{W}~\mbox{and}~ Y(A\tilde{W})^{\tilde{k}+1} =A^{\dag}(A\tilde{W})^{\tilde{k}+1}.
\]
\item If $X=A^{\circleddagger,W}$ \cite{Ferreyra6,mos7}, then the solution of the corresponding system is uniquely determined as $Y=A^{\dag}A\tilde{W}A^{\circleddagger,W}\tilde{W}$:
\[
YAY=Y, AY=A\tilde{W}A^{\circleddagger,W}\tilde{W}~\mbox{and}~ Y(A\tilde{W})^{\tilde{k}+1} =A^{\dag}(A\tilde{W})^{\tilde{k}+1}.
\]
\item For $X=A^{\textcircled{w}_m,W}$ \cite{mos8}, then $Y=A^{\dag}A\tilde{W}A^{\textcircled{w}_m,W}\tilde{W}$ \cite{Gao14}, can be expressed uniquely for the systems:
\[
YAY=Y, AY=A\tilde{W}A^{\textcircled{w}_m,W}\tilde{W}~\mbox{and}~ Y(A\tilde{W})^{\tilde{k}+1} =A^{\dag}(A\tilde{W})^{\tilde{k}+1}.
\]
\end{itemize}
In the below example, we illustrate the uniqueness of the W-weighted weak MPD inverse of the matrix $A$, as established in Theorem \ref{Inverse}. 

\begin{example}
Let us consider the matrix \\
    \[
    A=\begin{bmatrix}
    &1&1&0&1&\\&0&1&0&0&\\&0&0&0&1&\\&0&0&1&0&\\&0&0&0&0&\\
    \end{bmatrix} ~\mbox{and}~ \tilde{W}=\begin{bmatrix}
        &1&0&0&0&1&\\&0&0&1&0&0&\\&0&1&0&0&0&\\&0&0&0&0&0&\\
    \end{bmatrix}
    \]
    for which $ind(A\tilde{W}) = \tilde{k} = 3$.
    Since $X$ is minimal rank W-weighted weak Drazin inverse of $A$, we consider the matrix 
    $X=\begin{bmatrix}
        &1&x_1&0&x_2&\\&0&0&0&0&\\&0&0&0&0&\\&0&0&0&0&\\&0&0&0&0&\\
    \end{bmatrix}$ which satisfies \[ X=(A\tilde{W})^3[(A\tilde{W})^3]^{\dag}X~\mbox{and}~X\tilde{W}(A\tilde{W})^4=(A\tilde{W})^3~[\textnormal{Theorem~\ref{main1.1}}~\textnormal{(iv)}]. \]
    Now,
        $Y=A^{\dag}A\tilde{W}X\tilde{W}
    = \begin{bmatrix}
        &1&0&x_1&0&1&\\&0&0&0&0&0&\\&0&0&0&0&0&\\&0&0&0&0&0&\\
    \end{bmatrix}$, satisfies the conditions of Theorem \ref{Inverse} i.e., $YAY=Y$, $AY=A\tilde{W}X\tilde{W}$ and $Y(A\tilde{W})^4=A^{\dag}(A\tilde{W})^4$.

\begin{itemize}
    \item For $x_1=1~\mbox{and}~x_2=2$, the minimal rank W-weighted weak Drazin inverse coincides with W-weighted Drazin inverse \[A^{D,W} =[(A\tilde{W})^D]^2A=\begin{bmatrix}
        &1&1&0&2&\\&0&0&0&0&\\&0&0&0&0&\\&0&0&0&0&\\&0&0&0&0&\\
    \end{bmatrix}.\]
    \item If  $x_1=1$, then the W-weighted weak MPD inverse reduces to the W-weighted MPD inverse of $A$.   
\end{itemize} 
\end{example}
The next result establishes the characterizations of the W-weighted weak MPD inverse.
\begin{theorem} \label{Inverse 3.3} \normalfont
Suppose $B\in\mathbb{C}^{m\times{n}}$ and $\tilde{W}\in\mathbb{C}^{n\times{m}}$ such that $\tilde{k}=ind(B\tilde{W})$ and $X$ is a minimal rank W-weighted weak Drazin inverse of $B$. The following expressions are equivalent:\\
    (i) $Y=B^{\dag}B\tilde{W}X\tilde{W}$, i.e., $Y$ is W-weighted weak MPD ;\\
    (ii) $YBY=Y$, $BYB=B\tilde{W}X\tilde{W}B$, $BY=B\tilde{W}X\tilde{W}$ and $Y(B\tilde{W})^{\tilde{k}+1}=B^{\dag}(B\tilde{W})^{\tilde{k}+1}$ ;\\
    (iii) $YBY=Y$, $BY=B\tilde{W}X\tilde{W}$ and $YB=B^{\dag}BYB$ ;\\
    (iv) $BY=B\tilde{W}X\tilde{W}$, $YB=B^{\dag}BYB$ and $Y=B^{\dag}BY$ ;\\
     (v)  $YBY=Y$, $BY=B\tilde{W}X\tilde{W}$, $Y(B\tilde{W})^{\tilde{k}+1}=B^{\dag}(B\tilde{W})^{\tilde{k}+1}$ and $YX=B^{\dag}X$ ;\\
    (vi) $Y=B^{\dag}BY$, $BY=B\tilde{W}X\tilde{W}$ and $YXB^{\dag}=B^{\dag}XB^{\dag}$ ;\\
    (vii) $Y=B^{\dag}BY$, $BY=B\tilde{W}X\tilde{W}$ and $YXB=B^{\dag}XB$ .
    \end{theorem}
    \begin{proof}
    \begin{enumerate}
        \item [(ii)] $\Rightarrow$ (i): Given that $YBY=Y$, $BYB=B\tilde{W}X\tilde{W}B$, $BY=B\tilde{W}X\tilde{W}$ and $Y(B\tilde{W})^{\tilde{k}+1}=B^{\dag}(B\tilde{W})^{\tilde{k}+1}$. Now $Y=YBY=YB\tilde{W}X\tilde{W}=YB\tilde{W}(B\tilde{W})^{\tilde{k}}[(B\tilde{W})^{\tilde{k}}]^{\dag}XW=Y(B\tilde{W})^{\tilde{k}+1}[(B\tilde{W})^{\tilde{k}}]^{\dag}X\tilde{W}=B^{\dag}(B\tilde{W})^{\tilde{k}+1}[(B\tilde{W})^{\tilde{k}}]^{\dag}X\tilde{W}=B^{\dag}B\tilde{W}(B\tilde{W})^{\tilde{k}}[(B\tilde{W})^{\tilde{k}}]^{\dag}X\tilde{W}=B^{\dag}B\tilde{W}X\tilde{W}$.
        \item [(i)] $\Rightarrow$ (iii): Pre-multiplying $B^{\dag}B$ and post-multiplying $B$ to $Y=B^{\dag}B\tilde{W}X\tilde{W}$, we obtain \[B^{\dag}BYB=B^{\dag}BB^{\dag}B\tilde{W}X\tilde{W}B=B^{\dag}B\tilde{W}X\tilde{W}B=YB . \] The remaining two directly follow from Theorem \ref{Inverse}.
        \item [(iii)]$\Rightarrow$ (i): It is obvious.
        \item [(ii)] $\Rightarrow$ (iii): It is sufficient to verify the following matrix expression.
            \begin{eqnarray*}
            YB&=&YBYB
            =YB\tilde{W}X\tilde{W}B
            =YB\tilde{W}(B\tilde{W})^{\tilde{k}}[(B\tilde{W})^{\tilde{k}}]^{\dag}X\tilde{W}B\\
            &=&Y(B\tilde{W})^{\tilde{k}+1}[(B\tilde{W})^{\tilde{k}}]^{\dag}X\tilde{W}B
            =B^{\dag}(B\tilde{W})^{\tilde{k}+1}[(B\tilde{W})^{\tilde{k}}]^{\dag}X\tilde{W}B\\
            &=&B^{\dag}B\tilde{W}(B\tilde{W})^{\tilde{k}}[(B\tilde{W})^{\tilde{k}}]^{\dag}X\tilde{W}B
            =B^{\dag}B\tilde{W}X\tilde{W}B
            =B^{\dag}BYB.
            \end{eqnarray*}
        \item [(iii)] $\Rightarrow$ (ii): From $BY=B\tilde{W}X\tilde{W}$ and $YB=B^{\dag}BYB$, we have
            \begin{eqnarray*}
            Y(B\tilde{W})^{\tilde{k}+1}&=&YB\tilde{W}(B\tilde{W})^{\tilde{k}}
            =YB\tilde{W}X\tilde{W}(B\tilde{W})^{\tilde{k}+1}\\
            &=&B^{\dag}BYB\tilde{W}X\tilde{W}(B\tilde{W})^{\tilde{k}+1}
            =B^{\dag}B\tilde{W}X\tilde{W}B\tilde{W}X\tilde{W}(B\tilde{W})^{\tilde{k}+1}\\
            &=&B^{\dag}B\tilde{W}X\tilde{W}(B\tilde{W})^{\tilde{k}+1}
            =B^{\dag}B\tilde{W}(B\tilde{W})^{\tilde{k}}
            =B^{\dag}(B\tilde{W})^{\tilde{k}+1}.
            \end{eqnarray*}
        \item [(ii)]$\Rightarrow$ (iv) : Applying the matrix conditions $Y=YBY$, $BY=B\tilde{W}X\tilde{W}$ and $Y(B\tilde{W})^{\tilde{k}+1}=B^{\dag}(B\tilde{W})^{\tilde{k}+1}$ to the matrix expression $Y=YBY$, we deduce
            \begin{eqnarray*}
            Y&=&YBY
            =YB\tilde{W}X\tilde{W}
            =YB\tilde{W}(B\tilde{W})^{\tilde{k}}[(B\tilde{W})^{\tilde{k}}]^{\dag}X\tilde{W}
            =Y(B\tilde{W})^{\tilde{k}+1}[(B\tilde{W})^{\tilde{k}}]^{\dag}X\tilde{W}\\
            &=&B^{\dag}(B\tilde{W})^{\tilde{k}+1}[(B\tilde{W})^{\tilde{k}}]^{\dag}X\tilde{W}
            =B^{\dag}B\tilde{W}(B\tilde{W})^{\tilde{k}}[(B\tilde{W})^{\tilde{k}}]^{\dag}X\tilde{W}\\
            &=&B^{\dag}B\tilde{W}X\tilde{W}
            =B^{\dag}BY.
            \end{eqnarray*}
        \item [(iii)] $\Rightarrow$ (iv): Clearly, $YBY=YB\tilde{W}X\tilde{W}=B^{\dag}BYB\tilde{W}X\tilde{W}=B^{\dag}BYBY=B^{\dag}BY=Y$.\\
        \item [(iv)]$\Rightarrow$ (ii): Multiplying $BY=B\tilde{W}X\tilde{W}$ on the right-hand side of $Y$ yields
            \begin{eqnarray*}   
            YBY&=&YB\tilde{W}X\tilde{W}
            =B^{\dag}BYB\tilde{W}X\tilde{W}
            =B^{\dag}B\tilde{W}X\tilde{W}B\tilde{W}X\tilde{W}\\
            &=&B^{\dag}B\tilde{W}X\tilde{W}
            =B^{\dag}BY
            =Y. \end{eqnarray*}
        \item [(iv)]$\Rightarrow$ (iii): From the given conditions $BY=B\tilde{W}X\tilde{W}$ and $Y=B^{\dag}BY$, we derive $YBY=B^{\dag}BYBY=B^{\dag}B\tilde{W}X\tilde{W}B\tilde{W}X\tilde{W}=B^{\dag}B\tilde{W}X\tilde{W}=B^{\dag}BY=Y$.
        \item [(ii)]$\Rightarrow$ (v): We have $Y=YBY$, $BY=B\tilde{W}X\tilde{W}$ and $Y(B\tilde{W})^{\tilde{k}+1}=B^{\dag}(B\tilde{W})^{\tilde{k}+1}$, then
            \begin{eqnarray*}  
            YX&=&YBYX
            =YB\tilde{W}X\tilde{W}X
            =YB\tilde{W}(B\tilde{W})^{\tilde{k}}[(B\tilde{W})^{\tilde{k}}]^{\dag}X\tilde{W}X\\
            &=&Y(B\tilde{W})^{\tilde{k}+1}[(B\tilde{W})^{\tilde{k}}]^{\dag}X\tilde{W}X=B^{\dag}(B\tilde{W})^{\tilde{k}+1}[(B\tilde{W})^{\tilde{k}}]^{\dag}X\tilde{W}X\\
            &=&B^{\dag}B\tilde{W}(B\tilde{W})^{\tilde{k}}[(B\tilde{W})^{\tilde{k}}]^{\dag}X\tilde{W}X=B^{\dag}B\tilde{W}X\tilde{W}X
            =B^{\dag}X.
            \end{eqnarray*}
        \item [(ii)]$\Rightarrow$ (vi) : From $YBY=Y$, $BY=B\tilde{W}X\tilde{W}$ and $Y(B\tilde{W})^{\tilde{k}+1}=B^{\dag}(B\tilde{W})^{\tilde{k}+1}$, we obtain
            $YXB^{\dag}=YBYXB^{\dag}
            =YB\tilde{W}X\tilde{W}XB^{\dag}
            =YB\tilde{W}(B\tilde{W})^{\tilde{k}}[(B\tilde{W})^{\tilde{k}}]^{\dag}X\tilde{W}XB^{\dag}\\
            =Y(B\tilde{W})^{\tilde{k}+1}[(B\tilde{W})^{\tilde{k}}]^{\dag}X\tilde{W}XB^{\dag}
            =B^{\dag}(B\tilde{W})^{\tilde{k}+1}[(B\tilde{W})^{\tilde{k}}]^{\dag}X\tilde{W}XB^{\dag}\\
            =B^{\dag}B\tilde{W}(B\tilde{W})^{\tilde{k}}[(B\tilde{W})^{\tilde{k}}]^{\dag}X\tilde{W}XB^{\dag}
            =B^{\dag}B\tilde{W}X\tilde{W}XB^{\dag}
            =B^{\dag}XB^{\dag}$.
        \item [(vi)]$\Rightarrow$ (ii): Using the matrix expressions $Y=B^{\dag}BY$, $BY=B\tilde{W}X\tilde{W}$ and $YXB^{\dag}=B^{\dag}XB^{\dag}$, we get
            \[ YBY=B^{\dag}BYBY
            =B^{\dag}B\tilde{W}X\tilde{W}B\tilde{W}X\tilde{W}
            =B^{\dag}B\tilde{W}X\tilde{W}
            =B^{\dag}BY
            =Y~\mbox{and} \]
          \[ Y(B\tilde{W})^{\tilde{k}+1}=B^{\dag}BY(B\tilde{W})^{\tilde{k}+1}
            =B^{\dag}B\tilde{W}X\tilde{W}(B\tilde{W})^{\tilde{k}+1}
            =B^{\dag}B\tilde{W}(B\tilde{W})^{\tilde{k}}
            =B^{\dag}(B\tilde{W})^{\tilde{k}+1} . \]
        \item [(vi)]$\Rightarrow$ (iii): Consider the equations $Y=B^{\dag}BY$, $BY=B\tilde{W}X\tilde{W}$ and $YXB^{\dag}=B^{\dag}XB^{\dag}$. Using these equations, we establish
            \[YBY=B^{\dag}BYB\tilde{W}X\tilde{W}
            =B^{\dag}B\tilde{W}X\tilde{W}B\tilde{W}X\tilde{W}\\
            =B^{\dag}B\tilde{W}X\tilde{W}
            =B^{\dag}BY
            =Y.\]
        \item [(vi)]$\Rightarrow$ (v): Using the equations $Y=B^{\dag}BY$, $BY=B\tilde{W}X\tilde{W}$ and $YXB^{\dag}=B^{\dag}XB^{\dag}$, we derive
            \[YBY=B^{\dag}BYB\tilde{W}X\tilde{W}
            =B^{\dag}B\tilde{W}X\tilde{W}B\tilde{W}X\tilde{W}\\
            =B^{\dag}B\tilde{W}X\tilde{W}
            =B^{\dag}BY
            =Y. \]
            Further, it is shown that $Y(B\tilde{W})^{\tilde{k}+1}=B^{\dag}BY(B\tilde{W})^{\tilde{k}+1} =B^{\dag}B\tilde{W}X\tilde{W}(B\tilde{W})^{\tilde{k}+1}
            =B^{\dag}B\tilde{W}(B\tilde{W})^{\tilde{k}}
            =B^{\dag}(B\tilde{W})^{\tilde{k}+1}$~
             \mbox{and}~$YX=B^{\dag}BYX
            =B^{\dag}B\tilde{W}X\tilde{W}X
            =B^{\dag}X.$
    \end{enumerate}
    \end{proof}
We established the characterization for W-weighted weak DMP inverse.
    \begin{theorem} \label{Inverse 3.4} \normalfont
Suppose $A\in\mathbb{C}^{m\times{n}}$, $\tilde{W}\in\mathbb{C}^{n\times{m}}$ with $\tilde{k}=ind(\tilde{W}A)$ and $Z$ is a minimal rank W-weighted right weak Drazin inverse of $A$. The following expressions are equivalent:\\
    (i) $Y_1=\tilde{W}Z\tilde{W}AA^{\dag}$, i.e., $Y_1$ is W-weighted weak DMP ;\\
    (ii) $Y_1AY_1=Y_1$, $AY_1A=A\tilde{W}Z\tilde{W}A$, $Y_1A=\tilde{W}Z\tilde{W}A$ and $(\tilde{W}A)^{\tilde{k}+1}Y=(\tilde{W}A)^{\tilde{k}+1}A^{\dag}$ ;\\
    (iii) $Y_1AY_1=Y_1$, $Y_1A=\tilde{W}Z\tilde{W}A$ and $AY_1=AY_1AA^{\dag}$.\\
    (iv) $Y_1A=\tilde{W}Z\tilde{W}A$, $AY_1=AY_1AA^{\dag}$ and $Y_1=Y_1AA^{\dag}$ ;\\
     (v) $Y_1AY_1=Y_1$, $AY_1=A\tilde{W}Z\tilde{W}$, $(\tilde{W}A)^{\tilde{k}+1}Y_1=A^{\dag}(\tilde{W}A)^{\tilde{k}+1}$ and $ZY_1=ZA^{\dag}$ ;\\
    (vi) $Y_1=Y_1AA^{\dag}$, $Y_1A=\tilde{W}Z\tilde{W}A$ and $A^{\dag}ZY_1=A^{\dag}ZA^{\dag}$ ;\\
    (vii) $Y_1=Y_1AA^{\dag}$, $Y_1A=\tilde{W}Z\tilde{W}A$ and $AZY_1=AZA^{\dag}$.
    \end{theorem}
     If $X=A^{D,W}$, then the equivalent conditions of Theorem \ref{Inverse 3.3} satisfy the W-weighted MPD inverse, i.e., $Y=A^{\dag,D,W}$. 
     
    \begin{theorem} \normalfont
        Let us consider $B\in\mathbb{C}^{m\times{n}}$, $\tilde{W}\in\mathbb{C}^{n\times{m}}$ and $\tilde{k}=ind(B\tilde{W})$. Suppose $X$ is a minimal rank $W$-weighted weak Drazin inverse of $B$ with $X=B^{D,W}$. The following statements are equivalent:\\ \normalfont
    (i)  $Y=B^{\dag}B\tilde{W}B^{D,W}\tilde{W}$ ;\\
    (ii) $YBY=Y$, $BYB=B\tilde{W}B^{D,W}\tilde{W}B$, $BY=B\tilde{W}B^{D,W}\tilde{W}$ and $Y(B\tilde{W})^{\tilde{k}+1}=B^{\dag}(B\tilde{W})^{\tilde{k}+1}$ \cite{Kyrchei12};\\
    (iii) $YBY=Y$, $BY=B\tilde{W}B^{D,W}\tilde{W}$ and $YB=B^{\dag}BYB$ ;\\
    (iv) $BY=B\tilde{W}B^{D,W}\tilde{W}$, $YB=B^{\dag}BYB$ and $Y=B^{\dag}BY$ ;\\
    (v) $YBY=Y$, $BY=B\tilde{W}B^{D,W}\tilde{W}$, $Y(B\tilde{W})^{\tilde{k}+1}=B^{\dag}(B\tilde{W})^{\tilde{k}+1}$ and $YB^{D,W}=B^{\dag}B^{D,W}$ ;\\
    (vi) $Y=B^{\dag}BY$, $BY=B\tilde{W}B^{D,W}\tilde{W}$ and $YB^{D,W}B^{\dag}=B^{\dag}B^{D,W}B^{\dag}$ ;\\
    (vii) $Y=B^{\dag}BY$, $BY=B\tilde{W}B^{D,W}\tilde{W}$ and $YB^{D,W}B=B^{\dag}B^{D,W}B$ .
    \end{theorem}
The following results present the W-weighted DMP and MPD inverses through the projections of certain ranges and null spaces.

    \begin{lemma} \label{main 3.6} \normalfont
        Suppose $B\in\mathbb{C}^{m\times{n}}$ and $\tilde{W}\in\mathbb{C}^{n\times{m}}$ such that $\tilde{k}=ind(B\tilde{W})$. If $X$ is minimal rank W-weighted weak Drazin inverse of $B$ and $Y=B^{\dag}B\tilde{W}X\tilde{W}$, then
    \end{lemma}
    \begin{enumerate}
        \item[(i)] $BY=P_{\mathcal{R}{{((B\tilde{W})^{\tilde{k}}})}_,\mathcal{N}(X\tilde{W})}$ ;
        \item[(ii)] $YB=P_{\mathcal{R}{({B^{\dag}(B\tilde{W})^{\tilde{k}+1}})}_, \mathcal{N}(X\tilde{W}B)}$ ;
        \item[(iii)] $Y={B^{(2)}}_{\mathcal{R}{({B^{\dag}(B\tilde{W})^{\tilde{k}+1}})}_,\mathcal{N}(X\tilde{W})}$.
    \end{enumerate}
    \begin{proof}
    \begin{enumerate}
        \item[(i)] By Theorem \ref{Inverse}, we have $BY=B\tilde{W}X\tilde{W}$ which implies $\mathcal{R}(BY)=\mathcal{R}(B\tilde{W}X\tilde{W})$. Now, we need to prove $\mathcal{R}(B\tilde{W}X\tilde{W})=\mathcal{R}((B\tilde{W})^{\tilde{k}})$. Suppose  $x\in\mathcal{R}(B\tilde{W}X\tilde{W})$, i.e., \begin{eqnarray*}
             x=B\tilde{W}X\tilde{W}y &=& B\tilde{W}(B\tilde{W})^{\tilde{k}}[(B\tilde{W})^{\tilde{k}}]^{\dag}X\tilde{W}y~~[\textnormal {Theorem~\ref{main1.1}}~\textnormal{(iv)}]\\
            &=&(B\tilde{W})^{\tilde{k}}B\tilde{W}[(B\tilde{W})^{\tilde{k}}]^{\dag}X\tilde{W}y=(B\tilde{W})^{\tilde{k}}z.
        \end{eqnarray*} Hence $\mathcal{R}(B\tilde{W}X\tilde{W}) \subseteq \mathcal{R}((B\tilde{W})^{\tilde{k}})$. Similarly, by using Theorem \ref{main1.1} (vi), we obtain $\mathcal{R}((B\tilde{W})^{\tilde{k}}) \subseteq \mathcal{R}(B\tilde{W}X\tilde{W})$. This implies $\mathcal{R}(B\tilde{W}X\tilde{W})=\mathcal{R}((B\tilde{W})^{\tilde{k}})$. So $\mathcal{R}(BY)=\mathcal{R}(B\tilde{W}X\tilde{W})=\mathcal{R}((B\tilde{W})^{\tilde{k}})$.
        Moreover, considering $\mathcal{N}(BY)=\mathcal{N}(B\tilde{W}X\tilde{W})$ and  by Theorem  \ref{main1.1} (v), we obtain  
        $\mathcal{N}(B\tilde{W}X\tilde{W})=\mathcal{N}(X\tilde{W})$.
        So $\mathcal{N}(BY)=\mathcal{N}(X\tilde{W})$.
        \item[(ii)] We have $\mathcal{R}(YB)=\mathcal{R}(B^{\dag}B\tilde{W}X\tilde{W}B)$, since $Y=B^{\dag}B\tilde{W}X\tilde{W}$. Now, we will verify the condition $\mathcal{R}(B^{\dag}B\tilde{W}X\tilde{W}B) \subseteq \mathcal{R}(B^{\dag}(B\tilde{W})^{\tilde{k}+1})$ by using Theorem \ref{main1.1} (iv). Let $p\in\mathcal{R}(B^{\dag}B\tilde{W}X\tilde{W})$ that implies \begin{align*}
            p=B^{\dag}B\tilde{W}X\tilde{W}q=B^{\dag}B\tilde{W}(B\tilde{W})^{\tilde{k}}[(B\tilde{W})^{\tilde{k}}]^{\dag}X\tilde{W}Bq\\=B^{\dag}(B\tilde{W})^{\tilde{k}+1}[(B\tilde{W})^{\tilde{k}}]^{\dag}X\tilde{W}Bq=B^{\dag}(B\tilde{W})^{\tilde{k}+1}u.
        \end{align*} Similarly, $\mathcal{R}(B^{\dag}(B\tilde{W})^{\tilde{k}+1}) \subseteq \mathcal{R}(B^{\dag}B\tilde{W}X\tilde{W}B)$. Hence
            $$\mathcal{R}(B^{\dag}B\tilde{W}X\tilde{W}B)=\mathcal{R}(B^{\dag}(B\tilde{W})^{\tilde{k}+1}) = \mathcal{R}(YB).$$
        Furthermore, $\mathcal{N}(YB)=\mathcal{N}(B^{\dag}B\tilde{W}X\tilde{W}B)=\mathcal{N}(B\tilde{W}X\tilde{W}B)=\mathcal{N}(X\tilde{W}B)$.
        \item [(iii)] The proof follows directly from the equality: $\mathcal{R}(Y)=\mathcal{R}(YB)=\mathcal{R}(B^{\dag}B\tilde{W}X\tilde{W}B)=\mathcal{R}(B^{\dag}(B\tilde{W})^{\tilde{k}+1})$ and $\mathcal{N}(Y)=\mathcal{N}(BY)=\mathcal{N}(B\tilde{W}X\tilde{W})=\mathcal{N}(X\tilde{W})$.
        \end{enumerate}
    \end{proof}
    \begin{lemma}\label{main1.4} \normalfont
        Suppose $B\in\mathbb{C}^{m\times{n}}$ and $\tilde{W}\in\mathbb{C}^{n\times{m}}$ with $\tilde{k}=ind(\tilde{W}B)$. If $Z$ is minimal rank W-weighted right weak Drazin inverse of $B$ and $Y_1=\tilde{W}Z\tilde{W}BB^{\dag}$, then
    \end{lemma}
    \begin{enumerate}
        \item[(i)] $BY_1=P_{\mathcal{R}{{(B\tilde{W}Z)}}_,\mathcal{N}((\tilde{W}B)^{\tilde{k}+1}B^{\dag})}$ ;
        \item[(ii)] $Y_1B=P_{\mathcal{R}{(\tilde{W}Z)}_, \mathcal{N}((\tilde{W}B)^{\tilde{k}})}$ ;
        \item[(iii)] $Y_1={B^{(2)}}_{\mathcal{R}{(\tilde{W}Z)}_,   {\mathcal{N}((\tilde{W}B)^{\tilde{k}+1}B^{\dag})}}.$
    \end{enumerate}
    \begin{theorem} \normalfont
    Consider $B\in\mathbb{C}^{m\times{n}}$, $\tilde{W}\in\mathbb{C}^{n\times{m}}$ and $\tilde{k}=ind(B\tilde{W})$. Suppose $X$ is a minimal rank W-weighted weak Drazin inverse of $B$. Then $Y=B^{\dag}B\tilde{W}X\tilde{W}$ is the unique solution to 
        \[
        BY=P_{\mathcal{R}{{((BW)^{\tilde{k}}})}_,\mathcal{N}(X\tilde{W})}~~\mbox{and}~~\mathcal{R}(Y) \subseteq \mathcal{R}(B^*). 
        \]
        \end{theorem}
        \begin{proof}
                 Since $Y=B^{\dag}B\tilde{W}X\tilde{W}$,  it implies $\mathcal{R}(Y) \subseteq \mathcal{R}(B^{\dag}) = \mathcal{R}(B^*) = \mathcal{R}(B^{\dag}B)$. If $Y$, $\tilde{Y}$ are two solutions, then using Lemma \ref{main 3.6}, $BY-B\tilde{Y}$=$P_{\mathcal{R}{{((B\tilde{W})^{\tilde{k}}})}_,\mathcal{N}(X\tilde{W})}$-$P_{\mathcal{R}{{((B\tilde{W})^{\tilde{k}}})}_,\mathcal{N}(X\tilde{W})}=\mathbf{O}$. It follows that $\mathcal{R}(Y-\tilde{Y}) \subseteq \mathcal{N}(B)=\mathcal{N}(B^{\dag}B)$. In a similar manner, it follows that $\mathcal{R}(\tilde{Y}) \subseteq \mathcal{R}(B^{\dag}B)$. Therefore, $\mathcal{R}(Y-\tilde{Y}) \subseteq
                 \mathcal{R}(B^{\dag}B) \cap \mathcal{N}(B^{\dag}B)$.
        \end{proof}
        \begin{lemma}\normalfont
            Suppose $B\in\mathbb{C}^{m\times{n}}$ and $\tilde{W}\in\mathbb{C}^{n\times{m}}$ with $\tilde{k}=ind(\tilde{W}B)$. If $Z$ is a minimal rank W-weighted right weak Drazin inverse of $B$, then $Y_1=\tilde{W}Z\tilde{W}BB^{\dag}$ is the unique solution to the system,
        \[
        Y_1B=P_{\mathcal{R}{{(\tilde{W}Z)}}_,\mathcal{N}{((\tilde{W}B)^{\tilde{k}})}}~\mbox{and}~\mathcal{R}(Y_1^*) \subseteq \mathcal{R}{(B)}.
        \]
        \end{lemma}
\subsection{Expressions for W-weighted weak MPD and DMP inverses}
In this subsection we obtain the expression of W-weighted weak MPD and DMP expressed through  W-weighted MPD and DMP inverses.
\begin{lemma} \label{Inverse 3.10} \normalfont
    Suppose $B\in\mathbb{C}^{m\times{n}}$, $\tilde{W}\in\mathbb{C}^{n\times{m}}$ and $\tilde{k}=ind(B\tilde{W})$. Let $X$ be the minimal rank W-weighted weak Drazin inverse of $B$. Then
    \[
    B^{\dag}B\tilde{W}X\tilde{W}=B^{\dag,D,W}B\tilde{W}X\tilde{W}.
    \]
\end{lemma}
\begin{proof}
    Assume $X$ be the minimal rank W-weighted weak Drazin inverse of $B$. By Theorem \ref{main1.1} (iv), we derive
    \begin{eqnarray*}
        B^{\dag}B\tilde{W}X\tilde{W}&=&B^{\dag}B\tilde{W}(B\tilde{W})^{\tilde{k}}[(B\tilde{W})^{\tilde{k}}]^{\dag}X\tilde{W}\\
        &=&B^{\dag}B\tilde{W}B^{D,\tilde{W}}\tilde{W}(B\tilde{W})^{\tilde{k}+1}[(B\tilde{W})^{\tilde{k}}]^{\dag}X\tilde{W}\\
       &=&B^{\dag}B\tilde{W}B^{D,\tilde{W}}\tilde{W}B\tilde{W}(B\tilde{W})^k[(B\tilde{W})^k]^{\dag}X\tilde{W}\\
        &=&B^{{\dag},D,\tilde{W}}B\tilde{W}X\tilde{W}.
    \end{eqnarray*}
\end{proof}
\begin{lemma} \normalfont
    Consider $B\in\mathbb{C}^{m\times{n}}$ and $\tilde{W}\in\mathbb{C}^{n\times{m}}$ such that $\tilde{k}=ind(\tilde{W}B)$. If $Z$ is a minimal rank W-weighted right weak Drazin inverse of $B$, then
    \begin{center}
    $\tilde{W}Z\tilde{W}BB^{\dag}=\tilde{W}Z\tilde{W}BB^{D,{\dag},W}.$
    \end{center}
\end{lemma}
\begin{proof}
    This proof is similar to Lemma \ref{Inverse 3.10}.
\end{proof}
By considering $X$ as a minimal rank W-weighted weak Drazin inverse of $B$, we obtain some general forms of W-weighted weak MPD inverse by using $\{1\}$- inverse of $B$.
\begin{theorem}\normalfont
    Suppose $B\in\mathbb{C}^{m\times{n}}$, $\tilde{W}\in\mathbb{C}^{n\times{m}}$ with $\tilde{k}=ind(B\tilde{W})$ and $Q\in B\{1\}$. If $X$ is a minimal rank W-weighted weak Drazin inverse of $B$. The following conditions are equivalent:
    \end{theorem}
    \begin{enumerate}
        \item[(i)] $B^{\dag}B\tilde{W}X\tilde{W}=QB\tilde{W}X\tilde{W}$ ;
        \item[(ii)]$Q(B\tilde{W})^{\tilde{k}+1}=B^{\dag}(B\tilde{W})^{\tilde{k}+1}$ ;
        \item [(iii)]$\mathcal{R}(Q(B\tilde{W})^{\tilde{k}+1})=\mathcal{R}(B^{\dag}(B\tilde{W})^{\tilde{k}+1})$ ;
        \item[(iv)] $\mathcal{R}(Q(B\tilde{W})^{\tilde{k}+1})\subseteq \mathcal{R}(B^{\dag}(B\tilde{W})^{\tilde{k}+1})$ ;
        \item[(v)] $Q=B^{\dag}+U(I-(B\tilde{W})^{\tilde{k}}[(B\tilde{W})^{\tilde{k}}]^{\dag})$, for arbitrary $U\in\mathbb{C}^{n\times{m}}$.
        \end{enumerate}
         \begin{proof}
         \begin{enumerate}
             \item[(i)]$\Rightarrow$(ii): From the given condition $B^{\dag}B\tilde{W}X\tilde{W}=QB\tilde{W}X\tilde{W}$, we obtain $Q(B\tilde{W})^{\tilde{k}+1}=QB\tilde{W}(B\tilde{W})^{\tilde{k}}=QB\tilde{W}X\tilde{W}(B\tilde{W})^{\tilde{k}+1}=B^{\dag}B\tilde{W}X\tilde{W}(B\tilde{W})^{\tilde{k}+1}=B^{\dag}B\tilde{W}(B\tilde{W})^{\tilde{k}}=B^{\dag}(B\tilde{W})^{\tilde{k}+1}$.   
             \item [(ii)]$\Rightarrow$ (i): From the expression $Q(B\tilde{W})^{\tilde{k}+1}=B^{\dag}(B\tilde{W})^{\tilde{k}+1}$ and using Theorem \ref{main1.1} (iv), we derive $B^{\dag}B\tilde{W}X\tilde{W}=B^{\dag}B\tilde{W}(B\tilde{W})^{\tilde{k}}[(B\tilde{W})^{\tilde{k}}]^{\dag}X{\tilde{W}}$ = $B^{\dag}(B\tilde{W})^{\tilde{k}+1}[(B\tilde{W})^{\tilde{k}}]^{\dag}X{\tilde{W}}$\\ = $Q(B\tilde{W})(B\tilde{W})^{\tilde{k}}[(B\tilde{W})^{\tilde{k}}]^{\dag}X{\tilde{W}}=QB\tilde{W}X\tilde{W}$.
             \item [(iv)]$\Rightarrow$(i) : Using $\mathcal{R}(Q(B\tilde{W})^{\tilde{k}+1})\subseteq \mathcal{R}(B^{\dag}(B\tilde{W})^{\tilde{k}+1})$, we have\\$Q(B\tilde{W})^{\tilde{k}+1}=B^{\dag}(B\tilde{W})^{\tilde{k}+1}V,~\mbox{for~some} ~V\in\mathbb{C}^{m\times{n}}$. Now, $Q(B\tilde{W})^{\tilde{k}+1}=B^{\dag}(B\tilde{W})^{\tilde{k}+1}V=B^{\dag}BB^{\dag}(B\tilde{W})^{\tilde{k}+1}V=B^{\dag}BQ(B\tilde{W})^{\tilde{k}+1}=B^{\dag}BQB\tilde{W}(B\tilde{W})^{\tilde{k}}=B^{\dag}B\tilde{W}(B\tilde{W})^{\tilde{k}}=B^{\dag}(B\tilde{W})^{\tilde{k}+1}$. Moreover, \begin{eqnarray*}
                 B^{\dag}B\tilde{W}X\tilde{W}&=&B^{\dag}B\tilde{W} (B\tilde{W})^{\tilde{k}}[(B\tilde{W})^{\tilde{k}}]^{\dag}X\tilde{W}~~~~[\textnormal{Theorem}~\ref{main1.1}~ \textnormal{(iv)}]\\
                 &=&B^{\dag}(B\tilde{W})^{\tilde{k}+1}[(B\tilde{W})^{\tilde{k}}]^{\dag}X\tilde{W}\\
                 &=&Q(B\tilde{W})^{\tilde{k}+1}[(B\tilde{W})^{\tilde{k}}]^{\dag}X\tilde{W}\\
                 &=&QB\tilde{W}(B\tilde{W})^{\tilde{k}}[(B\tilde{W})^{\tilde{k}}]^{\dag}X\tilde{W}\\
                 &=&QB\tilde{W}X\tilde{W}.
                 \end{eqnarray*}
          \item [(ii)]$\Rightarrow$ (v): Suppose $Q$ satisfies the equation, $Q(B\tilde{W})^{\tilde{k}+1}=E~,~E\in\mathbb{C}^{n\times {m}}$ and by using the condition (i) we get $Q(B\tilde{W})^{\tilde{k}+1}\Leftrightarrow QB\tilde{W}X\tilde{W}(B\tilde{W})^{\tilde{k}+1} \Leftrightarrow B^{\dag}B\tilde{W}X\tilde{W}(B\tilde{W})^{\tilde{k}+1}\\ \Leftrightarrow B^{\dag}B\tilde{W}(B\tilde{W})^{\tilde{k}} \Leftrightarrow B^{\dag}(B\tilde{W})^{\tilde{k}+1}$. Thus, we conclude that a particular solution of the matrix equation is $B^{\dag}$. Solving the associated homogeneous equation $Q(B\tilde{W})^{\tilde{k}+1}=\mathbf{O}$, we find that the general solution of 
$Q(B\tilde{W})^{\tilde{k}+1}=B^{\dag}(B\tilde{W})^{\tilde{k}+1}$
is given by $Q = B^{\dag} + U\bigl(I - (B\tilde{W})^{\tilde{k}}[(B\tilde{W})^{\tilde{k}}]^{\dag}\bigr),$
where $U$ is an arbitrary matrix in $\mathbb{C}^{n\times m}$.
           \item [(v)]$\Rightarrow$(ii): This follows directly by post-multiplying $(B\tilde{W})^{\tilde{k}+1}$ to $Q=B^{\dag}+U(I-(B\tilde{W})^{\tilde{k}}[(B\tilde{W})^{\tilde{k}}]^{\dag})$.
          
         \end{enumerate}
    \end{proof}
    We derive the general form of the W-weighted weak DMP inverse, assuming that $Z$ is the minimal rank W-weighted right weak Drazin inverse by using $\{1\}$- inverse of $A$.
    \begin{lemma} \normalfont
        Let $B\in\mathbb{C}^{m\times {n}}$, $\tilde{W}\in\mathbb{C}^{n\times{m}}$ with $\tilde{k}=ind(\tilde{W}B)$, $Q\in B\{1\}$ and $Z$ be a minimal rank W-weighted right weak Drazin inverse of $B$. The subsequent conditions are equivalent:
        \begin{enumerate}
            \item [(i)] $\tilde{W}Z\tilde{W}BB^{\dag}=\tilde{W}Z\tilde{W}BQ$ ;
            \item [(ii)] $(\tilde{W}B)^{\tilde{k}+1}Q=(\tilde{W}B)^{\tilde{k}+1}B^{\dag}$ ;
            \item [(iii)] $\mathcal{N}((\tilde{W}B)^{\tilde{k}+1}Q)=\mathcal{N}((\tilde{W}B)^{\tilde{k}+1})B^{\dag})$ ;
            \item [(iv)] $\mathcal{N}((\tilde{W}B)^{\tilde{k}+1}Q) \subseteq \mathcal{N}((\tilde{W}B)^{\tilde{k}+1}B^{\dag})$ ;
            \item [(v)] $Q=B^{\dag}+(I-[(\tilde{W}B)^{\tilde{k}}]^{\dag}(\tilde{W}B)^{\tilde{k}})U$, for arbitrary $U\in\mathbb{C}^{n\times{m}}$ . 
        \end{enumerate}
    \end{lemma}
We investigate the general solution of the W-weighted weak MPD inverse.
    \begin{lemma}\normalfont
        Suppose $B\in\mathbb{C}^{m\times{n}}$ and $\tilde{W}\in\mathbb{C}^{n\times{m}}$ with $\tilde{k}=ind(B\tilde{W})$. Assume $X$ is the minimal rank W-weighted weak Drazin inverse of $B$. Then the general solution of the equation $Y(B\tilde{W})^{\tilde{k}+1}=B^{\dag}(B\tilde{W})^{\tilde{k}+1}$ is \[Y=B^{\dag}+Z(I-B\tilde{W}X\tilde{W}),\]
        where $Z\in\mathbb{C}^{n\times{m}}$ is arbitrary.
    \end{lemma}
    We can express the canonical form of W-weighted weak MPD by using decomposition of $A$ and $\tilde{W}$ which is presented in \cite{Ferreyra6}.
\begin{lemma} \label{Inverse 3.15} \normalfont
Suppose $B\in\mathbb{C}^{m\times{n}}$ and $\tilde{W}\in\mathbb{C}^{n\times{m}}$ such that $\tilde{k}=\max\{ind(B\tilde{W}),ind(\tilde{W}B)\}$ and $rank(B\tilde{W})^{\tilde{k}}=q$, we have
    \begin{center}
          $B=M
\begin{bmatrix}
B_1 & B_2 \\
0 & B_3  \\
\end{bmatrix}N^*$ and
$\tilde{W}=N
\begin{bmatrix}
\tilde{W_1} & \tilde{W_2} \\
0 & \tilde{W_3}  \\
\end{bmatrix}M^*,$
    \end{center}
    where $M\in\mathbb{C}^{m\times{m}}$,  $N\in\mathbb{C}^{n\times{n}}$ are unitary and $B_1,\tilde{W_1}\in\mathbb{C}^{q\times{q}}$ are nonsingular. The matrices $B_2\in\mathbb{C}^{q \times(m-q)}$, $B_3\in\mathbb{C}^{(m-q) \times(n-q)}$ and $\tilde{W}_3\in\mathbb{C}^{(n-q)\times(m-q)}$ such that $B_3\tilde{W}_3$ and $\tilde{W}_3B_3$ are nilpotent of indices $ind(B\tilde{W})~\mbox{and}~ ind(\tilde{W}B)$, respectively.
    \begin{theorem} \normalfont
       Suppose $X$ is a minimal rank W-weighted weak Drazin inverse of $B$ which is decomposed as
       \begin{center}
       $X=M
           \begin{bmatrix}
             (\tilde{W_1}B_1\tilde{W_1})^{-1} & X_2 \\
             0 & 0  \\  
           \end{bmatrix}N^*$ \cite{mos1}.
       \end{center} 
          \begin{eqnarray*}
             \text{Then, } Y&=&B^{\dag}B\tilde{W}X\tilde{W}\\&=&N\begin{bmatrix}
               B_1^*\Delta Q & B_1^*D+ B_1^*\Delta B_1\tilde{W_1}X_3\tilde{W_3}\\
               (I-B_3^{\dag}B_3)B_2^*\Delta Q & (I-B_3^{\dag}B_3)B_2^*D+(I-B_3^{\dag}B_3)B_2^*\Delta B_1\tilde{W_1}X_3\tilde{W_3}
           \end{bmatrix}M^*,
          \end{eqnarray*}
       where $X_2$ is arbitrary, $Q=B_1\tilde{W_1}(\tilde{W_1}B_1\tilde{W_1})^{-1}\tilde{W_1}$, $\Delta=(B_1B_1^*+B_2(I-B_3^{\dag}B_3)B_2^*)^{-1}$ and $D=\Delta B_1\tilde{W_1}(\tilde{W_1}B_1\tilde{W_1})^{-1}\tilde{W_2}$.
    \end{theorem}
    \begin{proof}
        By Lemma \ref{Inverse 3.15}, we have
        \begin{center}
          $B=M
\begin{bmatrix}
B_1 & B_2 \\
0 & B_3  \\
\end{bmatrix}N^*$,
$\tilde{W}=N
\begin{bmatrix}
\tilde{W_1} & \tilde{W_2} \\
0 & \tilde{W_3}  \\
\end{bmatrix}M^*$, 
$X=M
           \begin{bmatrix}
             (\tilde{W_1}B_1\tilde{W_1})^{-1} & X_2\\
             0 & 0  \\  
           \end{bmatrix}N^*.$
    \end{center} However,
    $B^{\dag}=N
    \begin{bmatrix}
        B_1^*\Delta & -B_1^*\Delta B_2B_3^{\dag}\\
        (I-B_3^{\dag}B_3)B_2^*\Delta & B_3^{\dag}
        (I-B_3^{\dag}B_3)B_2^*\Delta B_2B_3^{\dag}\\
    \end{bmatrix}M^*$ using Lemma 6 of \cite{Deng18}. Upon substitution of $B^{\dag}$, $Y$ gives the final result.
    \end{proof}
\end{lemma}
In this section we discussed perturbation analysis of generalized inverses. Such analysis was presented for Moore–Penrose inverse \cite{Ben2}, DMP inverse in \cite{(2020)}, and W-weighted Drazin inverse in \cite{Wei Y}. 
\subsection{Perturbation analysis of W-weighted weak DMP and MPD inverses}
Our first result concerns the derivation of a perturbation formula associated with the minimal rank 
W-weighted weak Drazin inverse.
\begin{theorem} \label{Inverse 3.17} \normalfont
Let us assume $B\in\mathbb{C}^{m\times{n}}$, $\tilde{W}\in\mathbb{C}^{n\times{m}}$ and $\tilde{k}=ind(B\tilde{W})$. Suppose $X$ be the minimal rank W-weighted weak Drazin inverse of $B$ and $\mathfrak{D}=B+E\in\mathbb{C}^{m\times{n}}$. If $\mathcal{R}(E\tilde{W})\subseteq \mathcal{R}((B\tilde{W})^k),\mathcal{R}((E\tilde{W})^*)\subseteq \mathcal{R}((X\tilde{W}B\tilde{W})^*)$ and $||\tilde{W}E\tilde{W}X||\mathrm{<} 1$, then $X(I+\tilde{W}E\tilde{W}X)^{-1}$ is a minimal rank W-weighted weak Drazin inverse of $\mathfrak{D}$. Moreover,
\begin{center}    $X(I+\tilde{W}E\tilde{W}X)^{-1}\tilde{W}\mathfrak{D}\tilde{W}=X\tilde{W}B\tilde{W}~\mbox{and}~\tilde{W}\mathfrak{D}\tilde{W}(I+X\tilde{W}E\tilde{W})^{-1}X=\tilde{W}B\tilde{W}X$.
\end{center}
\begin{proof}
Since $\mathcal{R}(E\tilde{W})\subseteq \mathcal{R}((B\tilde{W})^{\tilde{k}})=\mathcal{R}(X)$, it follows that
\[
E\tilde{W}=(B\tilde{W})^{\tilde{k}}U
=X\tilde{W}(B\tilde{W})^{\tilde{k}+1}U
=X\tilde{W}B\tilde{W}(B\tilde{W})^{\tilde{k}}U
=X\tilde{W}B\tilde{W}E\tilde{W}, \]
for some $U\in \mathbb{C}^{m\times{m}}$. Further, $E\tilde{W}=XK=B\tilde{W}X\tilde{W}XK=B\tilde{W}X\tilde{W}E\tilde{W}$, for $K\in \mathbb{C}^{n\times{m}}$. From the condition $\mathcal{R}((E\tilde{W})^*)\subseteq \mathcal{R}((X\tilde{W}B\tilde{W})^*)$, for some $V^*\in \mathbb{C}^{m\times{m}}$, we have
\[(E\tilde{W})^* = (X\tilde{W}B\tilde{W})^*V^*
=(V(X\tilde{W}B\tilde{W}))^* \] \[ ~\mbox{and}~ E\tilde{W} = VX\tilde{W}B\tilde{W}
= V(X\tilde{W}B\tilde{W}X)\tilde{W}B\tilde{W}
= (VX\tilde{W}B\tilde{W})X\tilde{W}B\tilde{W}
= E\tilde{W}X\tilde{W}B\tilde{W}.\]
Hence $\tilde{W}\mathfrak{D}\tilde{W}=\tilde{W}B\tilde{W}+\tilde{W}E\tilde{W}X\tilde{W}B\tilde{W}=(I+\tilde{W}E\tilde{W}X)\tilde{W}B\tilde{W}$ and $\tilde{W}\mathfrak{D}\tilde{W}=\tilde{W}B\tilde{W}(I+X\tilde{W}E\tilde{W})$. Upon substitution, we obtain $X(I+\tilde{W}E\tilde{W}X)^{-1}\tilde{W}\mathfrak{D}\tilde{W}=X\tilde{W}B\tilde{W}~\mbox{and}~\tilde{W}\mathfrak{D}\tilde{W}(I+X\tilde{W}E\tilde{W})^{-1}X=\tilde{W}B\tilde{W}X$. The assumption $||\tilde{W}E\tilde{W}X||\mathrm{<}1$ yields that $I+\tilde{W}E\tilde{W}X$ and $I+X\tilde{W}E\tilde{W}$ are non-singular. Now $\mathfrak{D}\tilde{W}=B\tilde{W}+B\tilde{W}X\tilde{W}E\tilde{W}=B\tilde{W}(I+X\tilde{W}E\tilde{W})$, so $rank((\mathfrak{D}\tilde{W}))=rank((B\tilde{W}))$. Similarly, it is shown that $rank((\mathfrak{D}\tilde{W})^{\tilde{k}})=rank((B\tilde{W})^{\tilde{k}})$. Further simplification of $X(I+\tilde{W}E\tilde{W}X)^{-1}\tilde{W}\mathfrak{D}\tilde{W}=X\tilde{W}B\tilde{W}$ follows that 

$\bigl(X(I+\tilde{W}E\tilde{W}X)^{-1}\tilde{W}\mathfrak{D}\tilde{W}-I\bigr)(\mathfrak{D}\tilde{W})^{\tilde{k}})$ \begin{align*}    
&=(X\tilde{W}B\tilde{W}-I)(\mathfrak{D}\tilde{W})^{\tilde{k}}\\
 &=(X\tilde{W}B\tilde{W}-I)(B\tilde{W})^{\tilde{k}}+(X\tilde{W}B\tilde{W}-I)((\mathfrak{D}\tilde{W})^{\tilde{k}}-(B\tilde{W})^{\tilde{k}})\\
&=(X\tilde{W}B\tilde{W}-I)(B\tilde{W})^{\tilde{k}}+(X\tilde{W}B\tilde{W}-I)E\tilde{W}R~;R\in\mathbb{C}^{m\times{m}}\\
&=(X\tilde{W}B\tilde{W}-I)(B\tilde{W})^{\tilde{k}}+X\tilde{W}B\tilde{W}E\tilde{W}R-E\tilde{W}R\\
&=(X\tilde{W}B\tilde{W}-I)(B\tilde{W})^{\tilde{k}}+E\tilde{W}R-E\tilde{W}R\\
&=(X\tilde{W}B\tilde{W}-I)(B\tilde{W})^{\tilde{k}}\\
 &=X\tilde{W}(B\tilde{W})^{\tilde{k}+1}-(B\tilde{W})^{\tilde{k}}\\
&=(B\tilde{W})^{\tilde{k}}-(B\tilde{W})^{\tilde{k}}\\
&= \mathbf{O}.
\end{align*}
        Hence $X(I+\tilde{W}E\tilde{W}X)^{-1}W(\mathfrak{D}\tilde{W})^{\tilde{k}+1}=(\mathfrak{D}\tilde{W})^{\tilde{k}}$ and
        \begin{center} $rank((\mathfrak{D}\tilde{W})^{\tilde{k}})=rank((B\tilde{W})^{\tilde{k}})=rank(X)=rank(X(I+\tilde{W}E\tilde{W}X)^{-1})$.
        \end{center}
        Thus, we conclude that $X(I+\tilde{W}E\tilde{W}X)^{-1}$ is the minimal rank W-weighted weak Drazin inverse of $\mathfrak{D}$.
    \end{proof}
    Similarly, we have the following theorem for the minimal rank W-weighted right weak Drazin inverse.
    \end{theorem}
    \begin{theorem} \normalfont
    Consider $B\in\mathbb{C}^{m\times{n}}$ and  $\tilde{W}\in\mathbb{C}^{n\times{m}}$ such that $\tilde{k}=ind(\tilde{W}B)$. Suppose $Z$ is a minimal rank W-weighted right weak Drazin inverse of $B$ and $\mathfrak{D}=B+E\in\mathbb{C}^{m\times{n}}$. If  $\mathcal{R}(\tilde{W}E)\subseteq \mathcal{R}(\tilde{W}B\tilde{W}Z),\mathcal{R}((\tilde{W}E)^*)\subseteq \mathcal{R}((\tilde{W}B)^{\tilde{k}+1})^*)$ and $||Z\tilde{W}E\tilde{W}||\mathrm{<} 1$, then $(I+Z\tilde{W}E\tilde{W})^{-1}Z$ is the minimal rank W-weighted right weak Drazin inverse of $\mathfrak{D}$. Moreover,
    \begin{center}
        $(I+Z\tilde{W}E\tilde{W})^{-1}Z\tilde{W}\mathfrak{D}\tilde{W}=Z\tilde{W}B\tilde{W}~\mbox{and}~\tilde{W}\mathfrak{D}\tilde{W}(I+Z\tilde{W}E\tilde{W})^{-1}Z=\tilde{W}B\tilde{W}Z$.
    \end{center}
    \end{theorem}
    In the following, we establish a perturbation result concerning the W-weighted weak MPD inverse.
    \begin{theorem} \label{Inverse 3.19} \normalfont
        Suppose $B\in\mathbb{C}^{m\times{n}}$, $\tilde{W}\in\mathbb{C}^{n\times{m}}$ and $\tilde{k}=ind(B\tilde{W})$. Assume that $X$ is a minimal rank W-weighted weak Drazin inverse of $B$, $Y=B^{\dag}B\tilde{W}X\tilde{W}$ and $\mathfrak{D}=B+E\in\mathbb{C}^{m\times{n}}$. If $\mathcal{R}(E)\subseteq \mathcal{R}(E\tilde{W}) \subseteq \mathcal{R}((B\tilde{W})^{\tilde{k}})$, $\mathcal{R}(E^*)\subseteq \mathcal{R}((E\tilde{W})^*)\subseteq \mathcal{R}((X\tilde{W}B\tilde{W})^*) \subseteq
        \mathcal{R}(B^*)$ and $\max\{||\tilde{W}E\tilde{W}X||,||B^{\dag}E||\}{<}1$, then
        \begin{center}
            $\mathfrak{D}^{\dag}X(I+\tilde{W}E\tilde{W}X)^{-1}\tilde{W}\mathfrak{D}\tilde{W}X=\mathfrak{D}^{\dag}X\tilde{W}\mathfrak{D}\tilde{W}(I+X\tilde{W}E\tilde{W})^{-1}X=(I+YE)^{-1}YX$,
        \end{center}
        where $X(I+\tilde{W}E\tilde{W}X)^{-1}$ is the minimal rank W-weighted weak Drazin inverse of $\mathfrak{D}$. Furthermore,
        \begin{center} $\mathfrak{D}\mathfrak{D}^{\dag}X(I+\tilde{W}E\tilde{W}X)^{-1}\tilde{W}\mathfrak{D}\tilde{W}X=BYX$
            \end{center}
            and
           \[ \frac{||YX||}{1+||YE||} \le ||\mathfrak{D}^{\dag}X(I+\tilde{W}E\tilde{W}X)^{-1}\tilde{W}\mathfrak{D}\tilde{W}X|| \le \frac{||YX||}{1-||YE||}.\]
\begin{proof}
We have $\mathcal{R}(E)\subseteq \mathcal{R}(E\tilde{W}) \subseteq \mathcal{R}((B\tilde{W})^{\tilde{k}})\subseteq \mathcal{R}(B)$ and $\mathcal{R}(E^*) \subseteq \mathcal{R}(B^*)$.\\
Thus we can write
\[\mathfrak{D}=B+E
=B+BB^{\dag}E
=B(I+B^{\dag}E)\] and
\[\mathfrak{D}^{\dag}=(B(I+B^{\dag}E))^{\dag}
=(I+B^{\dag}E)^{-1}B^{\dag}=B^{\dag}(I+EB^{\dag})^{-1}.\]
It follows that
\begin{eqnarray*}
\mathfrak{D}\mathfrak{D}^{\dag}&=&\mathfrak{D}(I+B^{\dag}E)^{-1}B^{\dag}
=BB^{\dag},
\end{eqnarray*}
similarly,
\begin{eqnarray*}
\mathfrak{D}^{\dag}&=&B^{\dag}(I+EB^{\dag})^{-1}\\
\mathfrak{D}^{\dag}\mathfrak{D}&=&B^{\dag}(I+EB^{\dag})^{-1}\mathfrak{D}\\
&=&B^{\dag}B.
\end{eqnarray*}
Since $\mathcal{R}(E\tilde{W}) \subseteq \mathcal{R}((B\tilde{W})^{\tilde{k}})$, for some $U \in\mathbb{C}^{m \times {m}}$, we have
$E\tilde{W}=(B\tilde{W})^{\tilde{k}}U
=X\tilde{W}(B\tilde{W})^{\tilde{k}+1}U
=B\tilde{W}X\tilde{W}X\tilde{W}(B\tilde{W})^{\tilde{k}+1}U
=B\tilde{W}X\tilde{W}(B\tilde{W})^{\tilde{k}}U=B\tilde{W}X\tilde{W}E\tilde{W}$. Also, since $\mathcal{R}(E) \subseteq \mathcal{R}(E\tilde{W})$, for some matrix $V \in\mathbb{C}^{m \times {n}}$, we obtain
\[
E=E\tilde{W}V=B\tilde{W}X\tilde{W}E\tilde{W}V=B\tilde{W}X\tilde{W}E.
\]
Moreover, $I+YE=I+B^{\dag}B\tilde{W}X\tilde{W}E=I+B^{\dag}E$ is non-singular, as $||B^{\dag}E||\mathrm{<}1$. Now, by Theorem \ref{main1.1} and Theorem \ref{Inverse 3.17},
\begin{eqnarray*}
(I+YE)^{-1}YX&=&(I+B^{\dag}E)^{-1}B^{\dag}B\tilde{W}X\tilde{W}X\\                
&=&\mathfrak{D}^{\dag}X\tilde{W}B\tilde{W}X\\
&=&\mathfrak{D}^{\dag}X(I+\tilde{W}E\tilde{W}X)^{-1}\tilde{W}\mathfrak{D}\tilde{W}X,
\end{eqnarray*} and \[
\mathfrak{D}^{\dag}X\tilde{W}B\tilde{W}X=\mathfrak{D}^{\dag}X\tilde{W}\mathfrak{D}\tilde{W}(I+X\tilde{W}E\tilde{W})^{-1}X.
\]
From the preceding expression, it implies that \[ \mathfrak{D}\mathfrak{D}^{\dag}X(I+\tilde{W}E\tilde{W}X)^{-1}\tilde{W}\mathfrak{D}\tilde{W}X=BB^{\dag}X\tilde{W}B\tilde{W}X=BB^{\dag}B\tilde{W}X\tilde{W}X=BYX.\]
Since $||B^{\dag}E||\mathrm{<}1$, we can write
\[
||(I+YE)^{-1}||\le \frac{1}{1-||YE||}.
\] Moreover, \[
||(I+YE)^{-1}YX|| \le \frac{||YX||}{1-||YE||}.
\]
Furthermore, \begin{eqnarray*}
    ||YX||&=&||(I+YE)(I+YE)^{-1}YX||\\&\le& [1+||YE||]~||(I+YE)^{-1}YX||\\
\end{eqnarray*}  which implies \[ \frac{||YX||}{1+||YE||} \le ||(I+YE)^{-1}YX||.\]
Therefore, \[
\frac{||YX||}{1+||YE||} \le ||\mathfrak{D}^{\dag}X(I+\tilde{W}E\tilde{W}X)^{-1}\tilde{W}\mathfrak{D}\tilde{W}X|| \le \frac{||YX||}{1-||YE||}.
\]
\end{proof}
    \end{theorem}
    The next result follows the perturbation analysis of W-weighted weak DMP inverse.
    \begin{theorem} \label{Inverse 3.20} \normalfont
        Let us consider $B\in\mathbb{C}^{m\times{n}}$ and $\tilde{W}\in\mathbb{C}^{n\times{m}}$ such that $\tilde{k}=ind(\tilde{W}B)$. Suppose $Z$ is a minimal rank W-weighted right weak Drazin inverse of $B$, $Y=\tilde{W}Z\tilde{W}BB^{\dag}$ and $\mathfrak{D}=B+E\in\mathbb{C}^{m\times{n}}$. If $\mathcal{R}(E)\subseteq \mathcal{R}(\tilde{W}E)\subseteq \mathcal{R}(\tilde{W}B\tilde{W}Z) \subseteq \mathcal{R}(B)$, $\mathcal{R}(E^*)\subseteq \mathcal{R}((\tilde{W}E)^*) \subseteq \mathcal{R}(((\tilde{W}B)^{\tilde{k}+1})^*)$ and $\max\{||Z\tilde{W}E\tilde{W}||,||EB^{\dag}||\} {<} 1$, then
        \begin{center}
            $Z\tilde{W}\mathfrak{D}\tilde{W}(I+Z\tilde{W}E\tilde{W})^{-1}Z\mathfrak{D}^{\dag}=Z(I+Z\tilde{W}E\tilde{W})^{-1}Z\tilde{W}\mathfrak{D}\tilde{W}Z\mathfrak{D}^{\dag}=ZY(I+EY)^{-1}$,
        \end{center}
        where $(I+Z\tilde{W}E\tilde{W})^{-1}Z$ is a minimal rank W-weighted right weak Drazin inverse of $\mathfrak{D}$. Furthermore,
        \begin{center}
            $Z\tilde{W}\mathfrak{D}\tilde{W}(I+Z\tilde{W}E\tilde{W})^{-1}Z\tilde{W}\mathfrak{D}\tilde{W}\mathfrak{D}^{\dag}=ZYB$
            \end{center}
            and
            \[
            \frac{||ZY||}{1+||EY||} \le ||Z\tilde{W}\mathfrak{D}\tilde{W}(I+Z\tilde{W}E\tilde{W})^{-1}Z\mathfrak{D}^{\dag}|| \le \frac{||ZY||}{1-||EY||}.
            \]
    \end{theorem}
    If we substitute $X=B^{D,W}$ and $Z=B^{D,W}$ in Theorem \ref{Inverse 3.19} and Theorem \ref{Inverse 3.20}, which reduces to the perturbation formula for W-weighted MPD and W-weighted DMP inverses, respectively.
    \begin{corollary} \normalfont
        Let $B\in\mathbb{C}^{m\times{n}}$, $\tilde{W}\in\mathbb{C}^{n\times{m}}$ with $\tilde{k}=ind(B\tilde{W})$ and $\mathfrak{D}=B+E\in\mathbb{C}^{m\times{n}}$. If $\mathcal{R}(E) \subseteq \mathcal{R}(E\tilde{W})\subseteq \mathcal{R}((B\tilde{W})^{\tilde{k}})$, $\mathcal{R}(E^*)\subseteq \mathcal{R}((E\tilde{W})^*)\subseteq \mathcal{R}((X\tilde{W}B\tilde{W})^*) \subseteq \mathcal{R}(B^*)$  and $\max\{||\tilde{W}E\tilde{W}B^{D,W}||,||B^{\dag}E||\}{<}1$, then
        \begin{eqnarray*}
            \mathfrak{D}^{{\dag},D,W}&=&(I+B^{{\dag},D,W}E)^{-1}B^{{\dag},D,W}\\
            &=&B^{{\dag},D,W}(I+EB^{{\dag},D,W})^{-1}.
        \end{eqnarray*}
        In addition, $\mathfrak{D}\mathfrak{D}^{{\dag},D,W}=BB^{{\dag},D,W}$, $\mathfrak{D}^{{\dag},D,W}\mathfrak{D}=B^{{\dag},D,W}B$ and
\[
\frac{\left\|B^{{\dag},D,W}\right\|}{1 + \left\|B^{{\dag},D,W}E\right\|}
\;\leq\;
\left\|\mathfrak{D}^{{\dag},D,W}\right\|
\;\leq\;
\frac{\left\|B^{{\dag},D,W}\right\|}{1 - \left\|B^{{\dag},D,W}E\right\|}.
\]
\end{corollary}

\begin{corollary} \normalfont
         Let $B\in\mathbb{C}^{m\times{n}}$, $\tilde{W}\in\mathbb{C}^{n\times{m}}$ with $\tilde{k}=ind(\tilde{W}B)$ and $\mathfrak{D}=B+E\in\mathbb{C}^{m\times{n}}$. If $\mathcal{R}(E)\subseteq \mathcal{R}(\tilde{W}E)\subseteq \mathcal{R}(\tilde{W}B\tilde{W}Z) \subseteq \mathcal{R}(B)$, $\mathcal{R}(E^*) \subseteq \mathcal{R}((\tilde{W}E)^*)\subseteq \mathcal{R}(((\tilde{W}B)^{\tilde{k}+1})^*)$ and $\max\{||B^{D,W}\tilde{W}B\tilde{W}||,||EB^{\dag}||\}{<}1$, then
        \begin{eqnarray*}
            \mathfrak{D}^{D,{\dag},W}&=&(I+B^{D,{\dag},W}E)^{-1}B^{D,{\dag},W}\\
            &=&B^{D,{\dag},W}(I+EB^{D,{\dag},W})^{-1}.
        \end{eqnarray*}
        In addition, $\mathfrak{D}\mathfrak{D}^{D,{\dag},W}=BB^{D,{\dag},W}$, $\mathfrak{D}^{D,{\dag},W}\mathfrak{D}=B^{D,{\dag},W}B$ and
\[
\frac{\left\|B^{D,{\dag},W}\right\|}{1 + \left\|EB^{D,{\dag},W}\right\|}
\;\leq\;
\left\|\mathfrak{D}^{D,{\dag},W}\right\|
\;\leq\;
\frac{\left\|B^{D,{\dag},W}\right\|}{1 - \left\|EB^{D,{\dag},W}\right\|}.
 \]
 \end{corollary}

\subsection{Reverse and  Forward order law for minimal rank W-weighted weak Drazin inverse} 
 In generalized inverses, the forward and reverse order laws are fundamental concepts. Based on \cite{Greville21}, few sufficient conditions were first introduced for the reverse order law of the Moore-Penrose inverse, and later \cite{Castro} provides some results on the forward order law of the Moore-Penrose inverse. Both of these order laws were discussed in Drazin inverse \cite{Ben2, 2016}. In \cite{Wang H1} the triple reverse order law for the Drazin inverse was described. In this section, we explore these laws for W-weighted weak Drazin and minimal rank W-weighted weak Drazin inverses.
\begin{theorem}[\cite{Ben2}] \normalfont 
    Let $A,B\in\mathbb{C}^{m\times{n}}$, $\tilde{W}\in\mathbb{C}^{n\times{m}}$ with $\tilde{k}=\max \{ind(\tilde{W}A),ind(A\tilde{W})\}$ and $(A\tilde{W})(B\tilde{W})=(B\tilde{W})(A\tilde{W})$, then
    \begin{enumerate}
        \item [(a)] $(A\tilde{W})^D(B\tilde{W})=(B\tilde{W})(A\tilde{W})^D$ and $(B\tilde{W})^D(A\tilde{W})=(A\tilde{W})(B\tilde{W})^D$,
        \item [(b)] $[(A\tilde{W})(B\tilde{W})]^D=(B\tilde{W})^D(A\tilde{W})^D=(A\tilde{W})^D(B\tilde{W})^D$.
    \end{enumerate}
\end{theorem}

The next result is based on the reverse order law of the W-weighted weak Drazin inverse.
\begin{theorem}\label{Inverse 3.25} \normalfont
    Suppose $A,B\in\mathbb{C}^{m\times{n}}$, $W\in\mathbb{C}^{n\times{m}}$ with $(A\tilde{W})(B\tilde{W})=(B\tilde{W})(A\tilde{W})$ and $\tilde{k}=\max\{ind(A\tilde{W}),ind(B\tilde{W}),ind(A\tilde{W}B\tilde{W})\}$. If $X_1$ is the W-weighted weak Drazin inverse of $A\tilde{W}B$, then
     \[
     X_1=Y_2\tilde{W}Z_1,
     \] where $Y_2$ and $Z_1$ are the W-weighted weak Drazin inverses of $B$ and $A$.
     \begin{proof}
         Since $X_1$ is the W-weighted weak Drazin inverse of $A\tilde{W}B$, it satisfies the condition \[ X_1\tilde{W}(A\tilde{W}B\tilde{W})^{\tilde{k}+1}=(A\tilde{W}B\tilde{W})^{\tilde{k}}. \]
         Similarly, $Y_2$ and $Z_1$ are W-weighted weak Drazin inverse of $B$ and $A$, respectively, we have
         \[
         Y_2\tilde{W}(B\tilde{W})^{\tilde{k}+1}=(B\tilde{W})^{\tilde{k}}~\mbox{and}~Z_1\tilde{W}(A\tilde{W})^{\tilde{k}+1}=(A\tilde{W})^{\tilde{k}}.        \]
         Let us consider $U=Y_2\tilde{W}Z_1$. Then
         \begin{eqnarray*}
        U\tilde{W}(A\tilde{W}B\tilde{W})^{\tilde{k}+1}&=&Y_2\tilde{W}Z_1\tilde{W}(A\tilde{W}B\tilde{W})^{\tilde{k}+1}\\
        &=&Y_2\tilde{W}Z_1\tilde{W}(A\tilde{W})^{\tilde{k}+1}(B\tilde{W})^{\tilde{k}+1}\\
        &=&Y_2\tilde{W}(A\tilde{W})^{\tilde{k}}(B\tilde{W})^{\tilde{k}+1}\\
        &=&Y_2\tilde{W}(B\tilde{W})^{\tilde{k}+1}(A\tilde{W})^{\tilde{k}}\\
        &=&(B\tilde{W})^{\tilde{k}}(A\tilde{W})^{\tilde{k}}\\
        &=&(A\tilde{W})^{\tilde{k}}(B\tilde{W})^{\tilde{k}}\\
        &=&(A\tilde{W}B\tilde{W})^{\tilde{k}}.
    \end{eqnarray*}
    Hence, $U$ satisfying the property of the W-weighted weak Drazin inverse of $A\tilde{W}B$. Therefore, $X_1=Y_2\tilde{W}Z_1$.    
     \end{proof}
\end{theorem}
Subsequently, the forward order law for the W-weighted weak Drazin inverse follows is presented.
\begin{theorem}\normalfont
    Suppose $A,B\in\mathbb{C}^{m \times{n}}$ and $\tilde{W} \in\mathbb{C}^{n \times {m}}$ with  $(A\tilde{W})(B\tilde{W})=(B\tilde{W})(A\tilde{W})$. If $X_12$ is the W-weighted weak Drazin inverse of $A\tilde{W}B$, then $X_{12}=Z_1\tilde{W}Y_2$, where $Z_1$ and $Y_2$ are the W-weighted weak Drazin inverses of $A$ and $B$.
\end{theorem}
Based on Theorem \ref{Inverse 3.25}, we introduced the rank condition and explored the reverse order laws for minimal rank W-weighted weak Drazin inverse.
\begin{theorem}\label{Inverse 3.27} \normalfont
     Suppose $A,B\in\mathbb{C}^{m\times{n}}$, $\tilde{W}\in\mathbb{C}^{n\times{m}}$ with $(A\tilde{W})(B\tilde{W})=(B\tilde{W})(A\tilde{W})$ and $\tilde{k}=\max\{ind(A\tilde{W}),ind(B\tilde{W}),ind(A\tilde{W}B\tilde{W})\}$. Suppose $X_2$, $Y_3$, $Z_2$ are minimal rank W-weighted weak Drazin inverses of $A\tilde{W}B$, $B$, $A$, respectively and $(Y_3\tilde{W})(A\tilde{W})=(A\tilde{W})(Y_3\tilde{W})$. Then
     \[
     X_2=Y_3\tilde{W}Z_2.
     \]
\end{theorem}
\begin{proof}
    Since $X_2$, $Y_3$ and $Z_2$ are minimal rank W-weighted weak Drazin inverses of $A\tilde{W}B$, $B$ and $A$, we have
    \[ X_2\tilde{W}(A\tilde{W}B\tilde{W})^{\tilde{k}+1}=(A\tilde{W}B\tilde{W})^{\tilde{k}},~ rank(X_2)=rank((A\tilde{W}B\tilde{W})^{\tilde{k}}),\]
   \[ Y_3\tilde{W}(B\tilde{W})^{\tilde{k}+1}=(B\tilde{W})^{\tilde{k}},~rank(Y_3)=rank((B\tilde{W})^{\tilde{k}}), \]
    and \[ Z_2\tilde{W}(A\tilde{W})^{\tilde{k}+1}=(A\tilde{W})^{\tilde{k}},~rank(Z_2)=rank(A\tilde{W})^{\tilde{k}}. \]
    Let us define $Q=Y_3\tilde{W}Z_2$, then
    \begin{eqnarray*}
        Q\tilde{W}(A\tilde{W}B\tilde{W})^{\tilde{k}+1}&=&Y_3\tilde{W}Z_2\tilde{W}(A\tilde{W}B\tilde{W})^{\tilde{k}+1}\\
        &=&Y_3\tilde{W}Z_2\tilde{W}(A\tilde{W})^{\tilde{k}+1}(B\tilde{W})^{\tilde{k}+1}\\
        &=&Y_3\tilde{W}(A\tilde{W})^{\tilde{k}}(B\tilde{W})^{\tilde{k}+1}\\
        &=&Y_3\tilde{W}(B\tilde{W})^{\tilde{k}+1}(A\tilde{W})^{\tilde{k}}\\
        &=&(B\tilde{W})^{\tilde{k}}(A\tilde{W})^{\tilde{k}}\\
        &=&(A\tilde{W})^{\tilde{k}}(B\tilde{W})^{\tilde{k}}\\
        &=&(A\tilde{W}B\tilde{W})^{\tilde{k}}.
    \end{eqnarray*}
    Thus, \[
        \mathcal{R}((A\tilde{W}B\tilde{W})^{\tilde{k}}) \subseteq \mathcal{R}(Q)~\mbox{and}~\mathcal{R}(Q) \subseteq \mathcal{R}((A\tilde{W}B\tilde{W})^{\tilde{k}}),
        \]
        which implies \[ rank((A\tilde{W}B\tilde{W})^{\tilde{k}})=rank(Q). \]
        Therefore,\[ rank(X_2)=rank((A\tilde{W}B\tilde{W})^{\tilde{k}})=rank(Y_3\tilde{W}Z_2). \]
    Hence $X_2$ and $Y_3\tilde{W}Z_2$ satisfy the conditions of the minimal rank W-weighted weak Drazin inverse of $A\tilde{W}B$. It follows that \[ X_2=Y_3\tilde{W}Z_2.\]
\end{proof}
Next, we include the forward order law for minimal rank W-weighted weak Drazin inverse.
\begin{theorem}\label{Inverse 3.28} \normalfont
     Suppose $A,B\in\mathbb{C}^{m\times{n}}$, $\tilde{W}\in\mathbb{C}^{n\times{m}}$ with $(A\tilde{W})(B\tilde{W})=(B\tilde{W})(A\tilde{W})$ and $\tilde{k}=\max\{ind(A\tilde{W}),ind(B\tilde{W}),ind(A\tilde{W}B\tilde{W})\}$. Suppose $X_3$, $Y_3$, $Z_2$ are minimal rank W-weighted weak Drazin inverses of $A\tilde{W}B$, $B$, $A$, respectively and $(Z_2\tilde{W})(B\tilde{W})=(B\tilde{W})(Z_2\tilde{W})$. Then
     \[
     X_3=Z_2\tilde{W}Y_3.
     \]
\end{theorem}
In Theorem \ref{Inverse 3.27}, if we fix $X_2$, $Y_3$ and $Z_2$ as W-weighted Drazin inverses of the respective matrices, then we obtain the following result for the reverse order law.
\begin{corollary} \normalfont
Let us consider $A,B\in\mathbb{C}^{m\times{n}}$, $\tilde{W}\in\mathbb{C}^{n\times{m}}$ with $(A\tilde{W})(B\tilde{W})=(B\tilde{W})(A\tilde{W})$ and $\tilde{k}=\max\{ind(A\tilde{W}),ind(B\tilde{W}),ind(A\tilde{W}B\tilde{W})\}$. If $X_2=(A\tilde{W}B)^{D,W}$, $Y_3=B^{D,W}$ and $Z_2=A^{D,W}$ are W-weighted Drazin inverse of $A\tilde{W}B$, $B$, $A$, respectively and $(B^{D,W}\tilde{W})(A\tilde{W})=(A\tilde{W})(B^{D,W}\tilde{W})$, then
     \[
     (A\tilde{W}B)^{D,W}=B^{D,W}\tilde{W}A^{D,W}.
     \]
\end{corollary}
Similarly, we present the forward order law by considering the W-weighted Drazin inverse of the appropriate matrices in Theorem \ref{Inverse 3.28}.
\begin{corollary}\normalfont
    Suppose $A,B\in\mathbb{C}^{m\times{n}}$, $\tilde{W}\in\mathbb{C}^{n\times{m}}$ with $(A\tilde{W})(B\tilde{W})=(B\tilde{W})(A\tilde{W})$ and $\tilde{k}=\max\{ind(A\tilde{W}),ind(B\tilde{W}),ind(A\tilde{W}B\tilde{W})\}$. Suppose we fix $X_3=(A\tilde{W}B)^{D,W}$, $Y_3=B^{D,W}$ and $Z_2=A^{D,W}$ are W-weighted Drazin inverses of $A\tilde{W}B$, $B$, and $A$, respectively. If $(A^{D,W}\tilde{W})(B\tilde{W})=(B\tilde{W})(A^{D,W}\tilde{W})$, then
     \[
     (A\tilde{W}B)^{D,W}=A^{D,W}\tilde{W}B^{D,W}.
     \]
\end{corollary}
\begin{itemize}
    \item In a similar way, we can establish the reverse order law for W-weighted core-EP inverse $A^{\circleddagger,W}$ and the W-weighted m-WGI $A^{\textcircled{w}_m,W}$.
\end{itemize}

    Further, we explore the triple reverse order law for minimal rank W-weighted weak Drazin inverse. 
\begin{theorem}\label{Inverse 3.31} \normalfont
    Let $A,B,C\in\mathbb{C}^{m\times{n}}$, $\tilde{W}\in\mathbb{C}^{n\times{m}}$ with $A\tilde{W}$, $B\tilde{W}$, $C\tilde{W}$ are commuting with each other and $\tilde{k}=\max\{ind(A\tilde{W}),ind(B\tilde{W}),ind(C\tilde{W}),ind(A\tilde{W}B\tilde{W}C\tilde{W})\}$. Suppose $X_4$, $Y_4$, $Z_3$ and $U_1$ are minimal rank W-weighted weak Drazin inverses of $A\tilde{W}B\tilde{W}C$, $B$, $A$ and $C$, respectively, with $(U_1\tilde{W})(A\tilde{W}B\tilde{W})=(A\tilde{W}B\tilde{W})(U_1\tilde{W})$. Then
     \[
     X_4=U_1\tilde{W}Y_4\tilde{W}Z_3.
     \]
\end{theorem}
\begin{proof}
    In Theorem \ref{Inverse 3.27}, we know that $X_2=Y_4\tilde{W}Z_3$.
    Since $X_4$, $U_1$ are minimal rank W-weighted weak Drazin inverse of $A\tilde{W}B\tilde{W}C$, $C$, we have
    \[
    X_4\tilde{W}(A\tilde{W}B\tilde{W}C\tilde{W})^{\tilde{k}+1}=(A\tilde{W}B\tilde{W}C\tilde{W})^{\tilde{k}},~ rank(X_3)=rank((A\tilde{W}B\tilde{W}C\tilde{W})^{\tilde{k}})
    \] and
    \[
    U_1\tilde{W}(C\tilde{W})^{\tilde{k}+1}=(C\tilde{W})^{\tilde{k}},~rank(U_1)=rank((C\tilde{W})^{\tilde{k}}).
    \]Consider $Q_1=U_1\tilde{W}Y_4\tilde{W}Z_3$. Then
    \begin{eqnarray*}
        Q_1\tilde{W}(A\tilde{W}B\tilde{W}C\tilde{W})^{\tilde{k}+1}&=&U_1\tilde{W}Y_4\tilde{W}Z_3\tilde{W}(A\tilde{W}B\tilde{W}C\tilde{W})^{\tilde{k}+1}\\
        &=&U_1\tilde{W}Y_4\tilde{W}Z_3\tilde{W}(A\tilde{W}B\tilde{W})^{\tilde{k}+1}(C\tilde{W})^{\tilde{k}+1}\\
        &=&U_1\tilde{W}(A\tilde{W}B\tilde{W})^{\tilde{k}}(C\tilde{W})^{\tilde{k}+1}\\
        &=&U_1\tilde{W}(C\tilde{W})^{\tilde{k}+1}(A\tilde{W}B\tilde{W})^{\tilde{k}}\\
        &=&(C\tilde{W})^{\tilde{k}}(A\tilde{W}B\tilde{W})^{\tilde{k}}\\
        &=&(A\tilde{W}B\tilde{W}C\tilde{W})^{\tilde{k}}.
    \end{eqnarray*}
    Hence \[
        \mathcal{R}((A\tilde{W}B\tilde{W}C\tilde{W})^{\tilde{k}}) \subseteq \mathcal{R}(Q_1)~\mbox{and}~\mathcal{R}(Q_1) \subseteq \mathcal{R}((A\tilde{W}B\tilde{W}C\tilde{W})^{\tilde{k}}),
        \]
    which implies \[ rank((A\tilde{W}B\tilde{W}C\tilde{W})^{\tilde{k}})=rank(Q_1). \]
 Therefore $Q_1$ and $U_1\tilde{W}Y_4\tilde{W}Z_3$ satisfies the conditions of the minimal rank W-weighted weak Drazin inverse of $A\tilde{W}B\tilde{W}C$. It follows that \[ X_4=U_1\tilde{W}Y_4\tilde{W}Z_3.\]
\end{proof}
We investigate the triple forward order law for minimal rank W-weighted weak Drazin inverse.
\begin{theorem}\label{Inverse 3.32} \normalfont
    Suppose $A,B\in\mathbb{C}^{m\times{n}}$, $\tilde{W}\in\mathbb{C}^{n\times{m}}$ with $A\tilde{W}$, $B\tilde{W}$, $C\tilde{W}$ are commuting with each other and $\tilde{k}=\max\{ind(A\tilde{W}),ind(B\tilde{W}),ind(C\tilde{W}),ind(A\tilde{W}B\tilde{W}C\tilde{W})\}$. Let $X_5$, $Y_4$, $Z_3$ and $U_1$ are minimal rank W-weighted weak Drazin inverses of $A\tilde{W}B\tilde{W}C$, $B$, $A$ and $C$, respectively, with $(Z_3\tilde{W}Y_4\tilde{W})(C\tilde{W})=(C\tilde{W})(Z_3\tilde{W}Y_4\tilde{W})$. Then 
     \[
     X_5=Z_3\tilde{W}Y_4\tilde{W}U_1.
     \]
\end{theorem}
We present an example illustrating both reverse and forward order laws applicable to the minimal rank W-weighted weak Drazin inverse, as well as the corresponding triple reverse and forward order laws.
\begin{example}
    Consider the following matrices $A=\begin{bmatrix}
        &1&1&1&1&\\&0&0&1&0&\\&0&1&0&1&\\&0&0&0&0&\\&0&0&0&0&\\
    \end{bmatrix}$, $B=\begin{bmatrix}
        &1&0&1&0&\\&0&1&0&0&\\&0&0&0&0&\\&0&0&0&0&\\&0&0&0&0&\\
    \end{bmatrix}$, $C=\begin{bmatrix}
        &1&1&0&1&\\&0&1&1&0&\\&0&0&0&0&\\&0&0&0&0&\\&0&0&0&0&\\
    \end{bmatrix}$ and $\tilde{W}=\begin{bmatrix}
        &1&1&0&1&1&\\&0&0&0&0&0&\\&0&0&0&0&0&\\&0&0&0&0&0&\\
    \end{bmatrix}$. These matrices satisfy the commutative properties: $(A\tilde{W})(B\tilde{W})=(B\tilde{W})(A\tilde{W})$ and $(A\tilde{W})(B\tilde{W})(C\tilde{W})=(C\tilde{W})(B\tilde{W})(A\tilde{W})$. Now  
    $Z_2=\begin{bmatrix}
        &1&z_1&z_2&z_3&\\&0&0&0&0&\\&0&0&0&0&\\&0&0&0&0&\\&0&0&0&0&\\
    \end{bmatrix},$ $Y_3=\begin{bmatrix}
        &1&y_1&0&y_2&\\&0&0&0&0&\\&0&0&0&0&\\&0&0&0&0&\\&0&0&0&0&\\
    \end{bmatrix}$ and $U_1=\begin{bmatrix}
        &1&u_1&0&0&\\&0&0&0&0&\\&0&0&0&0&\\&0&0&0&0&\\&0&0&0&0&\\
    \end{bmatrix}$ are minimal rank W-weighted weak Drazin inverses of $A$, $B$, and $C$, respectively. Further, $X_2=\begin{bmatrix}
        &1&z_1&z_2&z_3&\\&0&0&0&0&\\&0&0&0&0&\\&0&0&0&0&\\&0&0&0&0&\\ \end{bmatrix}$, $X_3=\begin{bmatrix}
        &1&y_1&0&y_2&\\&0&0&0&0&\\&0&0&0&0&\\&0&0&0&0&\\&0&0&0&0&\\ \end{bmatrix}$ are minimal rank W-weighted weak Drazin inverses of $A\tilde{W}B$ and $X_4=\begin{bmatrix}
        &1&z_1&z_2&z_3&\\&0&0&0&0&\\&0&0&0&0&\\&0&0&0&0&\\&0&0&0&0&\\ \end{bmatrix}$, $X_5=\begin{bmatrix}
        &1&u_1&0&0&\\&0&0&0&0&\\&0&0&0&0&\\&0&0&0&0&\\&0&0&0&0&\\ \end{bmatrix}$ are the minimal rank W-weighted weak Drazin inverses of $A\tilde{W}B\tilde{W}C$. By Theorem \ref{Inverse 3.27}, $Y_3$ satisfies the condition $(Y_3\tilde{W})(A\tilde{W})=(A\tilde{W})(Y_3\tilde{W})$, therefore $Y_3\tilde{W}Z_2=\begin{bmatrix}
        &1&z_1&z_2&z_3&\\&0&0&0&0&\\&0&0&0&0&\\&0&0&0&0&\\&0&0&0&0&\\
    \end{bmatrix} = X_2$. Similarly by Theorem \ref{Inverse 3.28}, $Z_2$ satisfies the condition $(Z_2\tilde{W})(B\tilde{W})=(B\tilde{W})(Z_2\tilde{W})$ and we derive $Z_2\tilde{W}Y_3=\begin{bmatrix}
        &1&y_1&0&y_2&\\&0&0&0&0&\\&0&0&0&0&\\&0&0&0&0&\\&0&0&0&0&\\
    \end{bmatrix} = X_3$. Now by Theorem \ref{Inverse 3.31} and Theorem \ref{Inverse 3.32}, we obtain the triple reverse and forward order laws for minimal rank W-weighted weak Drazin inverse as $X_4=U_1\tilde{W}Y_4\tilde{W}Z_3=\begin{bmatrix} &1&z_1&z_2&z_3&\\&0&0&0&0&\\&0&0&0&0&\\&0&0&0&0&\\&0&0&0&0&\\ 
    \end{bmatrix}$ and $X_5=Z_3\tilde{W}Y_4\tilde{W}U_1=\begin{bmatrix}
        &1&u_1&0&0&\\&0&0&0&0&\\&0&0&0&0&\\&0&0&0&0&\\&0&0&0&0&
        \end{bmatrix},$ respectively. 
\end{example}
The triple reverse and forward order laws are reduced to the following results for the W-weighted Drazin inverse.
\begin{corollary}\normalfont
    Suppose $A,B,C\in\mathbb{C}^{m\times{n}}$, $\tilde{W}\in\mathbb{C}^{n\times{m}}$ with $A\tilde{W}$, $B\tilde{W}$, $C\tilde{W}$ are commuting with each other and $\tilde{k}=\max\{ind(A\tilde{W}),ind(B\tilde{W}),ind(C\tilde{W}),ind(A\tilde{W}B\tilde{W}C\tilde{W})\}$. Assume that $X_4=(A\tilde{W}B\tilde{W}C)^{D,W}$, $Y_4=B^{D,W}$, $Z_3=C^{D,W}$ and $U_1=C^{D,W}$ are W-weighted Drazin inverses of $A\tilde{W}B\tilde{W}C$, $B$, $A$ and $C$, respectively, with $(C^{D,W}\tilde{W})(A\tilde{W}B\tilde{W})=(A\tilde{W}B\tilde{W})(C^{D,W}\tilde{W})$. Then
     \[
     (A\tilde{W}B\tilde{W}C)^{D,W}=C^{D,W}\tilde{W}B^{D,W}\tilde{W}A^{D,W}.
     \]
\end{corollary}
\begin{corollary}\normalfont
     Let us assume that $A,B,C\in\mathbb{C}^{m\times{n}}$, $\tilde{W}\in\mathbb{C}^{n\times{m}}$ with $A\tilde{W}$, $B\tilde{W}$, $C\tilde{W}$ are commuting with each other and $\tilde{k}=\max\{ind(A\tilde{W}),ind(B\tilde{W}),ind(C\tilde{W}),ind(A\tilde{W}B\tilde{W}C\tilde{W})\}$. Suppose $X_5=(A\tilde{W}B\tilde{W}C)^{D,W}$, $Y_4=B^{D,W}$, $Z_3=C^{D,W}$ and $U_1=C^{D,W}$ are W-weighted Drazin inverses of $A\tilde{W}B\tilde{W}C$, $B$, $A$ and $C$, respectively, with $(A^{D,W}\tilde{W}B^{D,W}\tilde{W})(C\tilde{W})=(C\tilde{W})(A^{D,W}\tilde{W}B^{D,W}\tilde{W})$. Then
     \[
     (A\tilde{W}B\tilde{W}C)^{D,W}=A^{D,W}\tilde{W}B^{D,W}\tilde{W}C^{D,W}.
     \]
\end{corollary}
We establish the reverse order law for the W-weighted weak MPD inverse.
\begin{theorem}\label{Inverse 3.30} \normalfont
    Let $A,B \in\mathbb{C}^{m \times{n}}$, $\tilde{W}\in\mathbb{C}^{n \times{m}}$ with $\tilde{k}=\max\{ind(A\tilde{W}),ind(B\tilde{W})\}$ and $(A\tilde{W})(B\tilde{W})=(B\tilde{W})(A\tilde{W})$. Suppose $X_2$, $Y_3$ and $Z_2$ are minimal rank W-weighted weak Drazin inverses of $A\tilde{W}B$, $B$ and $A$, respectively and that $\tilde{X}$, $\tilde{Y}_1$ and $\tilde{Z}$ are W-weighted weak MPD inverses of $A\tilde{W}B$, $B$ and $A$. If $\mathcal{R}(\tilde{W}BB^*\tilde{W}^*A^*) \subseteq \mathcal{R}(A^{\dag}A)$, $\tilde{W}^{\dag}\tilde{W}BB^*=BB^*\tilde{W}^{\dag}\tilde{W}$, ${\tilde{W}}^{\dag}A^{\dag}(B\tilde{W}Y_3\tilde{W})=(B\tilde{W}Y_3\tilde{W}){\tilde{W}}^{\dag}A^{\dag}$  and $(A\tilde{W})(Y_3\tilde{W})=(Y_3\tilde{W})(A\tilde{W})$, then
    \[
    \tilde{X}=\tilde{Y}_1A\tilde{Z}.
    \]
\end{theorem}
\begin{proof}
    Since $X_2$, $Y_3$ and $Z_2$ are minimal rank W-weighted weak Drazin inverses of $A\tilde{W}B$, $B$ and $A$, respectively, and $\tilde{X}$, $\tilde{Y}_1$ and $\tilde{Z}$ are the corresponding W-weighted weak MPD inverses of $A\tilde{W}B$, $B$ and $A$.
    The standard structural formula of the W-weighted weak MPD inverse is 
    \begin{equation} \label{3.1}
        \tilde{X}=(A\tilde{W}B)_W^{\dag,D_w}=(A\tilde{W}B)^{\dag}A\tilde{W}B\tilde{W}X_2\tilde{W}. 
    \end{equation} 
    We have $\mathcal{R}(\tilde{W}BB^*\tilde{W}^*A^*) \subseteq \mathcal{R}(A^{\dag}A)$, then \begin{eqnarray} \label{3.2}
        \tilde{W}BB^*\tilde{W}^*A^*&=&A^{\dag}AH,~\text{for some matrix H},\nonumber \\
        &=&A^{\dag}AA^{\dag}AH \nonumber \\
        &=&A^{\dag}A\tilde{W}BB^*\tilde{W}^*A^*. 
    \end{eqnarray}
    Multiplying on both sides of the equation (\ref{3.2}) by $B^{\dag}\tilde{W}^{\dag}$ on the pre-multiply side and by $(A\tilde{W}B)^{*\dag}$ on the post-multiply side, we obtain
    \begin{eqnarray*}
B^{\dag}\tilde{W}^{\dag}A^{\dag}A\tilde{W}BB^*\tilde{W}^*A^*(A\tilde{W}B)^{*\dag}&=&B^{\dag}\tilde{W}^{\dag}\tilde{W}BB^*\tilde{W}^*A^*(A\tilde{W}B)^{*\dag}\\
        &=&B^{\dag}BB^*\tilde{W}^{\dag}\tilde{W}\tilde{W}^*A^*(A\tilde{W}B)^{*\dag}\\
        &=&B^*\tilde{W}^*A^*(A\tilde{W}B)^{*\dag}\\
        &=&(A\tilde{W}B)^*(A\tilde{W}B)^{*\dag}\\
        &=&(A\tilde{W}B)^{\dag}(A\tilde{W}B).
    \end{eqnarray*} Hence
    \[
    B^{\dag}\tilde{W}^{\dag}A^{\dag}A\tilde{W}B(A\tilde{W}B)^*(A\tilde{W}B)^{*\dag}=B^{\dag}\tilde{W}^{\dag}A^{\dag}A\tilde{W}B=(A\tilde{W}B)^{\dag}(A\tilde{W}B).
    \]
    From Theorem \ref{Inverse 3.27}, substitute $X_2=Y_1\tilde{W}Z_2$ in the matrix equation (\ref{3.1}). Then
    \begin{eqnarray*}
        \tilde{X}&=&(A\tilde{W}B)_W^{\dag,D_w}\\&=&(A\tilde{W}B)^{\dag}A\tilde{W}B\tilde{W}X_2\tilde{W}\\&=&B^{\dag}{\tilde{W}}^{\dag}A^{\dag}A\tilde{W}B\tilde{W}Y_3\tilde{W}Z_2\tilde{W}\\&=&B^{\dag}{\tilde{W}}^{\dag}A^{\dag}(B\tilde{W}Y_3\tilde{W})A\tilde{W}Z_2\tilde{W}\\&=&B^{\dag}B\tilde{W}Y_3\tilde{W}{\tilde{W}}^{\dag}A^{\dag}A\tilde{W}Z_2\tilde{W}\\&=&B_W^{\dag,D_w}{\tilde{W}}^{\dag}A_W^{\dag,D_w}\\&=&\tilde{Y}_1\tilde{W}^{\dag}\tilde{Z}.
    \end{eqnarray*}

\end{proof}
Following the above theorem, we discuss the reverse order law for W-weighted MPD inverse.
\begin{corollary}\normalfont
   Assume that $A,B \in\mathbb{C}^{m \times{n}}$, $\tilde{W}\in\mathbb{C}^{n \times{m}}$ such that $\tilde{k}=\max\{ind(A\tilde{W}),ind(B\tilde{W})\}$ and $(A\tilde{W})(B\tilde{W})=(B\tilde{W})(A\tilde{W})$. Suppose $X_2=A^{D,W}$, $Y_3=B^{D,W}$ and $Z_2=A^{D,W}$ are W-weighted Drazin inverses of $A\tilde{W}B$, $B$ and $A$, respectively and that $\tilde{X}=(A\tilde{W}B)^{\dag,D,W}$, $\tilde{Y}_1=B^{\dag,D,W}$ and $\tilde{Z}=A^{\dag,D,W}$ are W-weighted MPD inverses of $A\tilde{W}B$, $B$ and $A$. If $\mathcal{R}(\tilde{W}BB^*\tilde{W}^*A^*) \subseteq \mathcal{R}(A^{\dag}A)$, $\tilde{W}^{\dag}\tilde{W}BB^*=BB^*\tilde{W}^{\dag}\tilde{W}$, $(A\tilde{W})(B\tilde{W})^D=(B\tilde{W})^D(A\tilde{W})$, and ${\tilde{W}}^{\dag}A^{\dag}(B\tilde{W}B^{D,W}\tilde{W})=(B\tilde{W}B^{D,W}\tilde{W}){\tilde{W}}^{\dag}A^{\dag}$, then
    \[
    (A\tilde{W}B)^{\dag,D,W}=B^{\dag,D,W}{\tilde{W}}^{\dag}A^{\dag,D,W}.
    \]
\end{corollary}

By applying reverse and triple reverse order laws to $A\tilde{W}B$ and $A\tilde{W}B\tilde{W}C$, we derive the solution of a new class of matrix equation.
\begin{theorem} \normalfont
    Let $A,B\in\mathbb{C}^{m \times {n}}$, $B_1,\tilde{W}\in\mathbb{C}^{n \times {m}}$ with $(A\tilde{W})(B\tilde{W})=(B\tilde{W})(A\tilde{W})$ and $\tilde{k}=\max\{ ind(A\tilde{W}),ind(B\tilde{W}),ind(A\tilde{W}B\tilde{W})\}$. Suppose $X_2$, $Y_3$ and $Z_2$ are the minimal rank W-weighted weak Drazin inverse of $A\tilde{W}B$, $B$, and $A$ respectively with $(Y_3\tilde{W})(A\tilde{W})=(A\tilde{W})(Y_3\tilde{W})$. If $\mathcal{Y}$ is a W-weighted weak MPD inverse of $A\tilde{W}B$, then the general solution of the equation $\mathcal{Y}(A\tilde{W}B\tilde{W})^{\tilde{k}+1}=B_1(A\tilde{W}B\tilde{W})^{\tilde{k}}$ is given by $\mathcal{Y}=B_1X_2\tilde{W}+Z(I-A\tilde{W}B\tilde{W}X_2\tilde{W})$, where $Z\in\mathbb{C}^{n \times {m}}$ is any arbitrary matrix.
\end{theorem}
\begin{proof}
    Since $X_2$ is the minimal rank W-weighted weak Drazin inverse of $A\tilde{W}B$ i.e., it satisfies the conditions $X_2\tilde{W}(A\tilde{W}B\tilde{W})^{\tilde{k}+1}=(A\tilde{W}B\tilde{W})^{\tilde{k}}$ and $rank(X_2)=rank((A\tilde{W}B\tilde{W})^{\tilde{k}})$. Also by Theorem \ref{Inverse 3.27}, $X_2$ satisfies the reverse order law i.e., $X_3=Y_3\tilde{W}Z_2$.
    Now \begin{eqnarray*}
    B_1X_2\tilde{W}&=&B_1Y_3\tilde{W}Z_2\tilde{W},~~\mbox{by~Theorem \ref{Inverse 3.27}}\ \\
    &=&B_1(A\tilde{W}B\tilde{W})^{\tilde{k}}[(A\tilde{W}B\tilde{W})^{\tilde{k}}]^{\dag}Y_3\tilde{W}Z_2\tilde{W},~~\mbox{by~Theorem \ref{main1.1}}\\
    &=&\mathcal{Y}(A\tilde{W}B\tilde{W})^{\tilde{k}+1}[(A\tilde{W}B\tilde{W})^{\tilde{k}}]^{\dag}Y_3\tilde{W}Z_2\tilde{W}\\
    &=&\mathcal{Y}(A\tilde{W}B\tilde{W})(A\tilde{W}B\tilde{W})^{\tilde{k}}[(A\tilde{W}B\tilde{W})^{\tilde{k}}]^{\dag}Y_3\tilde{W}Z_2\tilde{W}\\
    &=&\mathcal{Y}A\tilde{W}B\tilde{W}Y_3\tilde{W}Z_2\tilde{W}\\
    &=&\mathcal{Y}A\tilde{W}B\tilde{W}X_2\tilde{W}.   
    \end{eqnarray*}
    Consequently, \begin{eqnarray*}
        \mathcal{Y}&=&\mathcal{Y}+B_1X_2\tilde{W}-\mathcal{Y}A\tilde{W}B\tilde{W}X_2\tilde{W}\\
        &=&B_1X_2\tilde{W}+\mathcal{Y}-\mathcal{Y}A\tilde{W}B\tilde{W}X_2\tilde{W}\\
        &=&B_1X_2\tilde{W}+\mathcal{Y}(I-A\tilde{W}B\tilde{W}X_2\tilde{W}),
        \end{eqnarray*} which implies that $\mathcal{Y}$ has the form \[ \mathcal{Y}=B_1X_2\tilde{W}+Z(I-A\tilde{W}B\tilde{W}X_2\tilde{W}). \]
\end{proof}
\begin{theorem} \normalfont
    Suppose $A,B,C\in\mathbb{C}^{m \times {n}}$, $B_1,\tilde{W}\in\mathbb{C}^{n \times {m}}$ with $A\tilde{W}$, $B\tilde{W}$ and $C\tilde{W}$ are commuting with each other and $\tilde{k}=\max\{ ind(A\tilde{W}),ind(B\tilde{W}), ind(C\tilde{W}),ind(A\tilde{W}B\tilde{W}C\tilde{W})\}$. Let $X_4$, $Y_4$, $Z_3$ and $U_1$ are the minimal rank W-weighted weak Drazin inverses of $A\tilde{W}B\tilde{W}C$, $B$, $A$, $C$, and respectively, with $(U_1\tilde{W})(A\tilde{W}B\tilde{W})=(A\tilde{W}B\tilde{W})(U_1\tilde{W})$. If $\mathcal{Y}$ is a W-weighted weak MPD inverse of $A\tilde{W}B\tilde{W}C$, the general solution of the equation \[ \mathcal{Y}(A\tilde{W}B\tilde{W}C\tilde{W})^{\tilde{k}+1}=B_1(A\tilde{W}B\tilde{W}C\tilde{W})^{\tilde{k}} \] is given by \[ \mathcal{Y}=B_1X_4\tilde{W}+Z(I-A\tilde{W}B\tilde{W}C\tilde{W}X_4\tilde{W}), \] where $Z\in\mathbb{C}^{n \times {m}}$ is arbitrary.
\end{theorem}
We examine the solutions of a new class of matrix equations by applying the forward and triple forward order laws to $A\tilde{W}B$ and $A\tilde{W}B\tilde{W}C$.
\begin{theorem} \normalfont
    Assume that $A,B\in\mathbb{C}^{m \times {n}}$, $B_1,\tilde{W}\in\mathbb{C}^{n \times {m}}$ with $(A\tilde{W})(B\tilde{W})=(B\tilde{W})(A\tilde{W})$ and $\tilde{k}=\max\{ind(A\tilde{W}),ind(B\tilde{W}),ind(A\tilde{W}B\tilde{W})\}$. Suppose $X_3$, $Y_3$ and $Z_2$ are the minimal rank W-weighted weak Drazin inverses of $A\tilde{W}B$, $B$, and $A$ respectively, with $(Y_3\tilde{W})(A\tilde{W})=(A\tilde{W})(Y_3\tilde{W})$. If $\mathcal{Y}$ is a W-weighted weak MPD inverse of $A\tilde{W}B$, then the general solution of the equation \[ \mathcal{Y}(A\tilde{W}B\tilde{W})^{\tilde{k}+1}=B_1(A\tilde{W}B\tilde{W})^{\tilde{k}} \] is given by \[ \mathcal{Y}=B_1X_3\tilde{W}+Z(I-A\tilde{W}B\tilde{W}X_3\tilde{W}), \] where $Z\in\mathbb{C}^{n \times {m}}$ is any arbitrary matrix.
\end{theorem}

\begin{theorem} \normalfont
    Let $A,B,C\in\mathbb{C}^{m \times {n}}$, $B_1,\tilde{W}\in\mathbb{C}^{n \times {m}}$ with $A\tilde{W}$, $B\tilde{W}$ and $C\tilde{W}$ are commuting with each other and $\tilde{k}=\max\{ ind(A\tilde{W}),ind(B\tilde{W}), ind(C\tilde{W}),ind(A\tilde{W}B\tilde{W}C\tilde{W})\}$. Assume that $X_5$, $Y_4$, $Z_3$ and $U_1$ are the minimal rank W-weighted weak Drazin inverse of $A\tilde{W}B\tilde{W}C$, $B$, $A$, $C$, and respectively, with $(Z_3\tilde{W}Y_4\tilde{W})(C\tilde{W})=(C\tilde{W})(Z_3\tilde{W}Y_4\tilde{W})$. If $\mathcal{Y}$ is a W-weighted weak MPD inverse of $A\tilde{W}B\tilde{W}C$, the general solution of the equation \[ \mathcal{Y}(A\tilde{W}B\tilde{W}C\tilde{W})^{\tilde{k}+1}=B_1(A\tilde{W}B\tilde{W}C\tilde{W})^{\tilde{k}} \] is given by \[ \mathcal{Y}=B_1X_5\tilde{W}+Z(I-A\tilde{W}B\tilde{W}C\tilde{W}X_5\tilde{W}), \] where $Z\in\mathbb{C}^{n \times {m}}$ is arbitrary.
\end{theorem}
\section{Conclusion}\label{conc1.1}
We have extended weak MPD and DMP inverses to rectangular matrices using minimal rank W-weighted Drazin inverse. Further, We have derived the explicit expression for W-weighted weak MPD and DMP inverses and shown their characterizations and representations along with the uniqueness for the solutions of some matrix equations.  In addition, we have established certain relationships between the perturbation matrix and the minimal rank W-weighted Drazin inverse, and derive perturbation formulas that include lower and upper bounds for specific matrices. Moreover, several sufficient conditions are explored for finding the reverse and forward order laws for W-weighted Drazin inverse and minimal rank W-weighted weak Drazin inverse of the matrices $A\tilde{W}B$ and $A\tilde{W}B\tilde{W}C$. As an application, we have proved some matrix equations using reverse and forward order laws.

The following work opens up several direction for future research:
\begin{itemize}
    \item Explore the reverse and forward order laws for the minimal rank W-weighted right weak Drazin inverse and W-weighted weak DMP inverse
    \item Investigate the W-weighted MP-m-WCI is represented by $B^{\dag,\textcircled{\#}_m,W}$ and is also defined as $B^{\dag,\textcircled{\#}_m,W}=B^{\dag}B\tilde{W}B^{\textcircled{\#}_m,W}\tilde{W}$, analogous to the dual of this matrix in terms of minimal rank W-weighted weak Drazin inverse.
\end{itemize}

\section{Declaration of competing interest} The authors assert that they have no conflict of interest related to this work.

\section{Data availability} There have been no data used in this research paper.


\begin{thebibliography}{99}
\bibitem{Ben2}
Ben-Israel, A., Greville, T. N., Generalized Inverses: Theory and Applications, New York, NY: Springer New York, (2003).
\bibitem{Campbell}
Campbell, S. L., Meyer, C. D., Generalized Inverses of Linear Transformations, SIAM, (2009).
\bibitem{Campbell20}
Campbell, S. L., Meyer Jr., C. D., Weak Drazin inverses, Linear Algebra Appl., 20(2), 167--178 (1978). \href{https://doi.org/10.1016/0024-3795(78)90048-4}{\textbf{doi:10.1016/0024-3795(78)90048-4}}
\bibitem{Castro}
Castro-Gonzalez, N., Hartwig, R., Perturbation results and the forward order law for the Moore-Penrose inverse of a product, Electron. J. Linear Algebra, 34, 514--525 (2018). \href{https://doi.org/10.13001/1081-3810.3365}{\textbf{doi:10.13001/1081-3810.3365}}
\bibitem{Cline3}
Cline, R. E., Greville, T. N. E., A Drazin inverse for rectangular matrices, Linear Algebra Appl., 29, 53--62 (1980). \href{https://doi.org/10.1016/0024-3795(80)90230-X}{\textbf{doi:10.1016/0024-3795(80)90230-X}}
\bibitem{Deng18}
Deng, C. Y., Du, H. K., Representations of the Moore-Penrose inverse of 2×2 block operator valued matrices, J.Korean Math. Soc., 46(6), 1139--1150 (2009).
\bibitem{Drazin13}
Drazin, M. P., Pseudo-inverses in associative rings and semigroups, Amer. Math. Monthly, 65(7), 506--514 (1958). \href{https://doi.org/10.1080/00029890.1958.11991949}{\textbf{doi:10.1080/00029890.1958.11991949}}
\bibitem{2021}
Ferreyra, D. E., Levis, F. E., Priori, A. N., Thome, N., The weak core inverse. Aequationes Math., 95(2), 351--373  (2021).
\bibitem{Ferreyra6}
Ferreyra, D. E., Levis, F. E., Thome, N., Revisiting the core EP inverse and its extension to rectangular matrices, Quaestiones Math., 41(2), 265--281 (2018). \href{https://doi.org/10.2989/16073606.2017.1377779}{\textbf{doi:10.2989/16073606.2017.1377779}}
\bibitem{Ferreyra9}
Ferreyra, D. E., Mosic, D., The $ W $-weighted $ m $-weak core inverse, arXiv preprint arXiv, 2403.14196 (2024).
\bibitem{Gao14}
Gao J, Zuo K, Wang QW, The W-weighted m-weak group MP inverse and its applications, arXiv preprint arXiv, 2411.01481 (2024).
\bibitem{Greville21}
Greville, T. N. E., Note on the generalized inverse of a matrix product, SIAM Rev., 8(4), 518--521 (1966). \href{https://doi.org/10.1137/100810}{\textbf{doi:10.1137/100810}}
\bibitem{Kyrchei12}
Kyrchei, I. I., Determinantal Representations of the Weighted Core‐EP, DMP, MPD, and CMP Inverses. J. Math., 9816038 (2020). \href{ https://doi.org/10.1155/2020/9816038}{\textbf{doi:10.1155/2020/9816038}}
\bibitem{Core}
Manjunatha Prasad, K., Mohana, K., Core–EP inverse, Linear Multilinear Algebra, 62(6), 792--802 (2014).
\bibitem{(2020)}
Ma, H., Gao, X., Stanimirović, P. S., Characterizations, iterative method, sign pattern and perturbation analysis for the DMP inverse with its applications, Appl. Math. Comput., 378, 125196 (2020). \href{https://doi.org/10.1016/j.amc.2020.125196}{\textbf{doi:10.1016/j.amc.2020.125196}}
\bibitem{Malik22}
Malik, S. B., Thome, N., On a new generalized inverse for matrices of an arbitrary index, Appl. Math. Comput., 226, 575--580 (2014). \href{https://doi.org/10.1016/j.amc.2013.10.060}{\textbf{doi:10.1016/j.amc.2013.10.060}}
\bibitem{Meng11}
Meng, L., The DMP inverse for rectangular matrices, Filomat, 31(19), 6015--6019 (2017). \href{DOI:10.2298/FIL1719015M}{\textbf{doi:10.2298/FIL1719015M}}
\bibitem{Philos}
Moore, E. H., General analysis, Pt. I, Mem. Amer. Philos. Soc. 1 (1935).
\bibitem{mos1}
Mosić, D., Minimal rank weighted weak Drazin inverses, Electron. J. Linear Algebra, 40, 714--728 (2024). \href{https://doi.org/10.13001/ela.2024.8825}{\textbf{doi:10.13001/ela.2024.8825}}
\bibitem{mos8}
Mosić, D., Stanimirović, P. S., Kazakovtsev, L. A., The $ m $-weak group inverse for rectangular matrices, Electron. Res. Arch., 31(3), (2024). \href{doi: 10.3934/era.2024083}{\textbf{doi:10.3934/era.2024083}}
\bibitem{mos16}
Mosić, D., Weak MPD and DMP inverses, J. Math. Anal. Appl., 540(2), 128653 (2024). \href{https://doi.org/10.1016/j.jmaa.2024.128653}{\textbf{doi:10.1016/j.jmaa.2024.128653}}
\bibitem{mos7}
Mosić, D., Weighted core–EP inverse of an operator between Hilbert spaces, Linear Multilinear Algebra, 67(2), 278--298 (2019). \href{https://doi.org/10.1080/03081087.2017.1418824}{\textbf{doi:10.1080/03081087.2017.1418824}}
\bibitem{New}
Mosić, D., Zhang, D., New representations and properties of the m-weak group inverse, Results Math., 78(3), 97 (2023). \href{DOI:10.1007/s00025-023-01878-7}{\textbf{10.1007/s00025-023-01878-7}}
\bibitem{Penrose17}
Penrose, R., A generalized inverse for matrices, Math. Proc. Cambridge Philos. Soc., 51(3), 406--413 (1955). \href{https://doi.org/10.1017/S0305004100030401}{\textbf{doi:10.1017/S0305004100030401}}
\bibitem{1956}
Rado, R., Note on generalized inverses of matrices, Math. Proc. Cambridge Philos. Soc., 52(3), 600--601 (1956).
\bibitem{mos5}
Stojanović, K. S., Mosić, D., Weighted MPCEP inverse of an operator between Hilbert spaces, Bull. Iran. Math. Soc., 48(1), 53--71 (2022).
\bibitem{Stank15}
Stanković, M. S., Djurić, M. V., Weighted weak Drazin inverses, Zb. rad. Filoz. fak. Nišu., Serija Matematik, 89--94 (1987).
\bibitem{2018}
Wang, H., Chen, J., Weak group inverse. Open Math., 16(1), 1218--1232 (2018). \href{DOI:10.1515/math-2018-0100}{\textbf{doi:10.1515/math-2018-0100}}
\bibitem{Wang H1}
Wang, H., Zhong, C. C., Triple reverse order law for the Drazin inverse, Appl. Math. J. Chin. Univ., 39(1), 55--68 (2024).
\bibitem{2016}
Wang, X., Yu, A., Li, T., Deng, C., Reverse order laws for the Drazin inverses, J. Math. Anal. Appl., 444(1), 672--689 (2016). \href{https://doi.org/10.1016/j.jmaa.2016.06.026}{\textbf{doi:10.1016/j.jmaa.2016.06.026}}
\bibitem{Wei Y}
Wei, Y., Woo, C. W., Lei, T., A note on the perturbation of the W-weighted Drazin inverse, Appl. Math. Comput., 149(2), 423--430 (2004). \href{https://doi.org/10.1016/S0096-3003(03)00150-4}{\textbf{doi:0.1016/S0096-3003(03)00150-4}}
\bibitem{Wu}
Wu, C., Chen, J., Minimal rank weak Drazin inverses: a class of outer inverses with prescribed range, Electron. J. Linear Algebra, 39, 1--16 (2023).
\bibitem{Ring}
Zhou, Y., Chen, J., Zhou, M., m-weak group inverses in a ring with involution, Rev. R. Acad. Cienc. Exactas Fís. Nat. Ser. A Mat., 115(1), 2 (2021). \href{DOI:10.1007/s13398-020-00932-1}{\textbf{doi:10.1007/s13398-020-00932-1}}














\end{thebibliography}
\end{document}